\documentclass[a4paper, 10pt, parskip=half]{scrartcl}

\usepackage[utf8]{inputenc}
\usepackage[T1]{fontenc}
\usepackage{lmodern}
\usepackage{csquotes}
\usepackage{amsmath}
\usepackage{amssymb}
\usepackage{amsthm}
\usepackage{amsfonts}
\usepackage{dsfont}
\usepackage{mathtools}
\usepackage{bbm}
\usepackage{marginnote}
\usepackage{caption}
\usepackage{subcaption}
\usepackage[dvipsnames]{xcolor}
\usepackage{tikz}
\usepackage{hyperref}
\hypersetup{%
  colorlinks=true,
  linkcolor=black,
  citecolor=black,
  urlcolor=blue,
  pdftitle={Finite-time blow-up in fully parabolic quasilinear Keller--Segel systems with supercritical exponents},
  pdfauthor={Xinru Cao and Mario Fuest},
  pdfkeywords={},
  bookmarksopen=true,
}

\usepackage[numbers, compress]{natbib}
\newcommand{\R}{\mathbb{R}}

\newcommand{\N}{\mathbb{N}}

\newcommand{\mc}[1]{\mathcal{#1}}

\newcommand{\ur}[1]{\mathrm{#1}}
\newcommand{\ure}{\ur{e}}

\ifdefined\labelenumi%
  \renewcommand{\labelenumi}{(\roman{enumi})}

\newcommand{\eps}{\varepsilon}
\newcommand{\vt}{\vartheta}

\newcommand{\gt}{>}
\newcommand{\lt}{<}

\newcommand{\defs}{\coloneqq}
\newcommand{\sfed}{\eqqcolon}

\newcommand{\ra}{\rightarrow}

\newcommand{\nea}{\nearrow}

\newcommand{\sea}{\searrow}

\newcommand{\ol}{\overline}
\newcommand{\wt}{\widetilde}

\newcommand{\dx}{\,\mathrm{d}x}

\newcommand{\dy}{\,\mathrm{d}y}

\newcommand{\dr}{\,\mathrm{d}r}

\newcommand{\dtau}{\,\mathrm{d}\tau}
\newcommand{\dsigma}{\,\mathrm{d}\sigma}
\newcommand{\drho}{\,\mathrm{d}\rho}

\newcommand{\ddt}{\frac{\mathrm{d}}{\mathrm{d}t}}

\newcommand{\hp}{\hphantom}
\newcommand{\pe}{\mathrel{\hp{=}}}

\newcommand{\tmax}{T_{\max}}

\newcommand{\intom}{\int_\Omega}

\newcommand{\Ombar}{\ol \Omega}

\newcommand{\leb}[2][\Omega]{\ensuremath{L^{#2}(#1)}}

\newcommand{\sob}[3][\Omega]{\ensuremath{W^{#2, #3}(#1)}}

\newcommand{\con}[2][\Ombar]{\ensuremath{C^{#2}(#1)}}

\newcommand{\pu}{\mathbbmss p}

\newcommand{\tops}{\texorpdfstring}

\newcommand{\maxmq}{\mathrm{M}(m, q)}
\newcommand{\onemaxmq}{1-\maxmq}

\newcommand{\nn}{\nonumber}
\newcommand{\into}{\int_{\Omega}}
\newcommand{\intb}{\int_{B_{r_0}}}
\newcommand{\intbr}{\int_0^{r_0}}
\newcommand{\intr}{\int_0^{R}}
\newcommand{\norm}[2][ ]{\|#2\|_{#1}}
\newcommand{\Om}{\Omega}
\newcommand{\F}{\mathcal{F}}
\newcommand{\D}{\mathcal{D}}

\newcommand{\pa}{\partial}

\makeatletter
\renewenvironment{proof}[1][\proofname]{\par
  \pushQED{\qed}%
  \normalfont \topsep0\p@\relax
  \trivlist
  \item[\hskip\labelsep\scshape
  #1\@addpunct{.}]\ignorespaces
}{%
  \popQED\endtrivlist\@endpefalse
}
\makeatother

\newtheorem{base}{Base}[section]
\numberwithin{equation}{section}

\newtheorem{theorem}[base]{Theorem} \newtheorem*{theorem*}{Theorem}
\newtheorem{lemma}[base]{Lemma} \newtheorem*{lemma*}{Lemma}
 \newtheorem*{prop*}{Proposition}
 \newtheorem*{cor*}{Corollary}

\theoremstyle{definition}
\newtheorem{remark}[base]{Remark} \newtheorem*{remark*}{Remark}
 \newtheorem*{definition*}{Definition}
 \newtheorem*{example*}{Example}
 \newtheorem*{cond*}{Condition}

\setkomafont{title}{\normalfont\Large}
\title{Finite-time blow-up in fully parabolic quasilinear Keller--Segel systems with supercritical exponents}
\usepackage{authblk}

\author[1]{Xinru Cao\footnote{e-mail: caoxinru@gmail.com}}
\author[2]{Mario Fuest\footnote{e-mail: fuest@ifam.uni-hannover.de, corresponding author}}

\affil[1]{School of Mathematics and Statistics, Donghua University, North Renmin Road 2999, 201620 Shanghai, China}
\affil[2]{Leibniz Universität Hannover, Institut für Angewandte Mathematik, Welfengarten 1, 30167 Hannover, Germany}
\begin{document}
\date{}
\maketitle

\KOMAoptions{abstract=true}
\begin{abstract}
\noindent We examine the possibility of finite-time blow-up of solutions to the fully parabolic quasilinear Keller--Segel model
\begin{align}\tag{$\star$}\label{prob:star}
  \begin{cases}
    u_t = \nabla \cdot ((u+1)^{m-1}\nabla u - u(u+1)^{q-1}\nabla v) & \text{in $\Omega \times (0, T)$}, \\
    v_t = \Delta v - v + u     & \text{in $\Omega \times (0, T)$}
    \end{cases}
\end{align}
in a ball $\Omega\subset \mathbb R^n$ with $n\geq 2$.
Previous results show that unbounded solutions exist for all $m, q \in \mathbb R$ with $m-q<\frac{n-2}{n}$, which, however, are necessarily global in time if $q \leq 0$.
It is expected that finite-time blow-up is possible whenever $q > 0$ but in the fully parabolic setting this has so far only been shown when $\max\{m, q\} \geq 1$.\\[0.2pt]
In the present paper, we substantially extend these findings.
Our main results for the two- and three-dimensional settings state that \eqref{prob:star} admits solutions blowing up in finite time if
\begin{align*}
  m-q<\frac{n-2}{n}
  \quad \text{and} \quad
  \begin{cases}
    q < 2m & \text{if } n = 2, \\
    q < 2m - \frac23 \text{ or } m > \frac23 & \text{if } n = 3,
  \end{cases}
\end{align*}
that is, also for certain $m, q$ with $\max\{m, q\} < 1$.
As a key new ingredient in our proof, we make use of (singular) pointwise upper estimates for $u$.
\\[0.5pt]
 \textbf{AMS Classification (2020):} {35B44 (primary); 35B33, 35K45, 35K59, 92C17 (secondary)}
\end{abstract}

\section{Introduction}
The delicate interplay between the smoothing effect of diffusion and the destabilizing nature of attractive cross-diffusion is illustrated by the quasilinear Keller--Segel system 
\begin{align}\label{prob:proto}
  \begin{cases}
    u_t = \nabla \cdot (\phi(u) \nabla u - \psi(u) \nabla v) & \text{in $\Omega \times (0, T)$}, \\
    v_t = \Delta v - v + u                                   & \text{in $\Omega \times (0, T)$}, \\
    \partial_\nu u = \partial_\nu v = 0                      & \text{on $\partial \Omega \times (0, T)$}, \\
    u(\cdot, 0) = u_0, v(\cdot, 0) = v_0                     & \text{in $\Omega$}
  \end{cases}
\end{align}
describing the evolution of cells (with density $u$) which are not only driven by random motion but are also attracted by a chemical signal (with density $v$) produced by themselves
(\cite{HillenPainterUserGuidePDE2009}).
When posed for instance in smooth, bounded domains $\Omega \subset \R^n$, $n \ge 2$, and when $\phi$ and $\psi$ are taken as the prototypical choices
\begin{align}\label{eq:intro:phi_psi_proto}
	\phi(s)=(s+1)^{m-1}\quad \text{and} \quad \psi(s)=s(s+1)^{q-1},
  \qquad \text{$s \ge 0$},
\end{align}
with $m, q \in \R$, the sign of $m - q -\frac{n-2}{n}$ turns out to be critical:
If $m-q>\frac{n-2}{n}$, solutions of \eqref{prob:proto} with reasonably smooth nonnegative initial data are always globally bounded
(\cite{TaoWinklerBoundednessQuasilinearParabolic2012}, \cite{IshidaEtAlBoundednessQuasilinearKeller2014}; see also \cite{HorstmannWinklerBoundednessVsBlowup2005}, \cite{SenbaSuzukiQuasilinearParabolicSystem2006}),
whereas if $m-q < \frac{n-2}{n}$ and $\Omega$ is a ball, unbounded solutions exist for initial data with arbitrarily small mass
(\cite{WinklerDoesVolumefillingEffect2009}, \cite{HorstmannWinklerBoundednessVsBlowup2005}).
(While smallness of $\intom u_0$ is hence not sufficient to guarantee global boundedness, smallness of initial data in certain spaces smaller than $\leb1$ is (\cite{ding2019sigma}).)
Moreover, these references entail similar dichotomies also for more general choices of $\phi$ and $\psi$,
which are then based on the growth rate of $\frac{\phi}{\psi}$, as long as $\phi$ decays at most algebraically in the limit of infinitely large population density.
(Faster decay rates are covered in, e.g., \cite{WinklerFamilyMasscriticalKellerSegel2022} and \cite{stinner2024critical}).

Unlike the minimal Keller--Segel model proposed in \cite{KellerSegelInitiationSlimeMold1970},
which is obtained upon taking $m = q = 1$ in \eqref{eq:intro:phi_psi_proto},
the quasilinear system \eqref{prob:proto} also accounts for so-called volume-filling effects,
i.e., the fact that the ability of a cell to move is hindered by nearby cells \cite{PainterHillenVolumefillingQuorumsensingModels2002}.
The most striking and well-known feature of said minimal Keller--Segel system is a critical mass phenomenon in the two-dimensional setting:
For initial data with small mass $\intom u_0$, the corresponding solution is globally bounded (\cite{NagaiEtAlApplicationTrudingerMoserInequality1997}),
while certain initial data with large mass lead to unbounded solutions (\cite{HorstmannWangBlowupChemotaxisModel2001}),
see also the survey \cite{BellomoEtAlMathematicalTheoryKeller2015}.

Investigating the conditions under which unbounded solutions fail to exist globally is interesting from both an application point of view and a mathematical perspective.
For the minimal Keller--Segel system, solutions blowing up in finite time have been constructed
both in the mass supercritical two-dimensional setting (\cite{MizoguchiWinklerBlowupTwodimensionalParabolic}, \cite{FuestLankeitCornersCollapseSimple2023})
and in the higher-dimensional case at arbitrary mass level (\cite{WinklerFinitetimeBlowupHigherdimensional2013}).
(For precedents regarding parabolic--elliptic simplifications, we refer to \cite{JagerLuckhausExplosionsSolutionsSystem1992}, \cite{NagaiBlowupRadiallySymmetric1995}, \cite{NagaiBlowupNonradialSolutions2001}.)
The question whether in the general supercritical case $m-q < \frac{n-2}{n}$ unbounded solutions of \eqref{prob:proto} may already blow up in finite time is less well understood.
While solutions are necessary global in time if $q \le 0$ and hence only infinite-time blow-up can happen in this setting (\cite{WinklerGlobalClassicalSolvability2019}), 
corresponding results for certain parabolic--elliptic counterparts of \eqref{prob:proto} (\cite{LankeitInfiniteTimeBlowup2020}, \cite{WinklerDjieBoundednessFinitetimeCollapse2010})
suggest that whenever $q > 0 $ (and $m - q < \frac{n-2}{n}$),
solutions blowing up in finite time should exist.
For the fully parabolic system, this conjecture has been partially answered under the additional assumption $\max\{m, q\} \ge 1$
(\cite{CieslakStinnerFinitetimeBlowupGlobalintime2012}, \cite{CieslakStinnerFiniteTimeBlowupSupercritical2014}, \cite{CieslakStinnerNewCriticalExponents2015}).
Moreover, \cite{LaurencotMizoguchiFiniteTimeBlowup2017} obtains finite-time blow-up for the critical case $m=2-\frac 2n$, $q=1$ (if $n \in \{3, 4\}$),
while \cite{CieslakLaurencotFiniteTimeBlowup2010} does so for $m \in [-1, 0)$, $q = 1$ and $n = 1$.
To the best of our knowledge, however, the possibility of finite-time blow-up for $\max\{m, q\} < 1$ has remained open for nearly a decade.

\paragraph{Main results.}
It is the purpose of the present paper to substantially extend these results:
We obtain solutions blowing up in finite time also for certain supercritical $m, q$ with $\max\{m, q\} < 1$.
\begin{theorem}\label{th:ftbu}
  Let $n \ge 2$, $R > 0$, $\Omega \defs B_R(0) \subset \R^n$,
  suppose that $m \in \R$ and $q > 0$ satisfy
  \begin{align}\label{eq:ftbu:cond_unbdd}
    m-q<\frac{n-2}{n}
  \end{align} 
  and
  \begin{align}\label{eq:ftbu:cond_main}
    (1 - \max\{m, q\}) \cdot \frac{n(n-1)}{(n(m-q) + 1)_+} < 2
  \end{align}
  (where $\xi_+ \defs \max\{\xi, 0\}$ for $\xi \in \R$ and with the convention that $a \cdot \frac{n(n-1)}{0} < 2$ for every $a \le 0$),
  and let $\phi$ and $\psi$ be as in \eqref{eq:intro:phi_psi_proto}.
  For every $M_u > 0$, there then exist nonnegative, radially symmetric initial data $u_0, v_0 \in \con\infty$ with $\intom u_0 = M_u$
  and a uniquely determined smooth classical solution $(u, v)$ of \eqref{prob:proto}
  which blows up at some finite time $\tmax$ in the sense that $\limsup_{t \nea \tmax} \|u(\cdot, t)\|_{\leb\infty} = \infty$.
\end{theorem}

\begin{remark}
  Actually, we can treat slightly more general functions $\phi$ and $\psi$.
  In the two-dimensional setting, one can in particular choose $\phi(s) = (s+1)^{m-1}$, $\psi(s) = s(s+1)^{q-1} \ln(s+\ure)$, $s \ge 0$ for $m=q>0$.
  We refer to Section~\ref{sec:phi_psi} for precise conditions.
\end{remark}

\begin{remark}
  The condition $q > 0$ in Theorem~\ref{th:ftbu} is necessary; if $q \le 0$, then all solutions are global (\cite{WinklerGlobalClassicalSolvability2019}).
  Likewise, at least up to equality, condition \eqref{eq:ftbu:cond_unbdd} is optimal as well; if $m - q > \frac{n-2}{n}$, then all solutions are not only global but also bounded (\cite{TaoWinklerBoundednessQuasilinearParabolic2012}).

  We note that \eqref{eq:ftbu:cond_main} is fulfilled whenever $\max\{m, q\} \ge 1$ and that we hence recover the main results from \cite{CieslakStinnerFinitetimeBlowupGlobalintime2012, CieslakStinnerFiniteTimeBlowupSupercritical2014, CieslakStinnerNewCriticalExponents2015}.
  The novelty of the present paper lies in the case $\max\{m, q\} < 1$.
  For $n \in \{2, 3\}$, Theorem~\ref{th:ftbu} in particular covers $m, q$ with
  \begin{align*}
    m - q < \frac{n-2}{n}
    \quad \text{and} \quad
    \begin{cases}
      q < 2m & \text{if } n = 2, \\
      q < 2m - \frac23 \text{ or } m > \frac23 & \text{if } n = 3.
    \end{cases}
  \end{align*}
  This is illustrated in Figure~\ref{fig:mq_graph} below.
\end{remark}

\vspace{-0.5em}
\begin{minipage}{\textwidth}
\begin{center}
\begin{minipage}{0.4\textwidth}
\begin{center}
\begin{tikzpicture}[scale=2]
\draw[->, very thick] (-1, 0) -- (1.4, 0) node[right]{$m$};
\draw[->, very thick] (0, -1) node[below=1em]{\large $n=2$} -- (0, 1.4) node[above]{$q$};
\draw[thick, fill=gray] (0, 0) -- (1, 1) -- (0.5, 1) -- cycle;
\draw[thick] (-0.05, 1) node[left]{$1$} -- (1, 1);
\draw[thick] (-0.9, 1) -- (-0.26, 1);
\draw[thick] (1, -0.05) node[below]{$1\vphantom{\frac11}$} -- (1, 0.05);
\draw[thick] (0.5, -0.05) node[below]{$\frac12$} -- (0.5, 0.05);
\draw[thick] (-0.9, -0.9) -- (1.2, 1.2);
\node[] at (0.6, -0.6) {GB};
\node[] at (0.6, 0.2) {GB};
\node[] at (-0.3, -0.6) {GB};
\node[] at (-0.55, -0.15) {IFTBU};
\node[] at (0.5, 1.2) {FTBU};
\node[] at (-0.5, 1.2) {FTBU};
\node[] at (0.2, 0.7) {?};
\node[] at (-0.5, 0.7) {?};
\end{tikzpicture}
\end{center}
\end{minipage}
\begin{minipage}{0.4\textwidth}
\begin{center}
\begin{tikzpicture}[scale=2]
\draw[->, very thick] (-1, 0) -- (1.4, 0) node[right]{$m$};
\draw[very thick] (0, -1) node[below=1em]{\large $n=3$} -- (0, -1/3);
\draw[->, very thick] (0, 0) -- (0, 1.4) node[above]{$q$};
\draw[thick, fill=gray] (1/3, 0) -- (1, 2/3) -- (1, 1) -- (2/3, 1) -- (2/3, 2/3) -- cycle;
\draw[thick] (-0.05, 1) node[left]{$1$} -- (1, 1);
\draw[thick] (-0.9, 1) -- (-0.26, 1);
\draw[thick] (-0.05, 2/3) node[left]{$\frac23$} -- (0.05, 2/3);
\draw[thick] (1, -0.05) node[below]{$1\vphantom{\frac11}$} -- (1, 0.05);
\draw[thick] (1/3, -0.05) node[below]{$\frac13$} -- (1/3, 0.05);
\draw[thick] (2/3, -0.05) node[below]{$\frac23$} -- (2/3, 0.05);
\draw[thick] (-0.9+1/3, -0.9) -- (1.2+1/3, 1.2);
\node[] at (5/6, -0.6) {GB};
\node[] at (5/6, 0.2) {GB};
\node[] at (-0.23, -0.5-1/3) {GB};
\node[] at (-0.55+1/3, -0.15) {IFTBU};
\node[] at (0.5, 1.2) {FTBU};
\node[] at (-0.5, 1.2) {FTBU};
\node[] at (0.3, 0.6) {?};
\node[] at (-0.5, 0.6) {?};
\end{tikzpicture}
\end{center}
\end{minipage}
\end{center}
\vspace{-1em}
\captionsetup{type=figure}
\captionof{figure}{Properties of \eqref{prob:proto} (with $\phi, \psi$ as in \eqref{eq:intro:phi_psi_proto}) for $n \in \{2, 3\}$ and different values of $m$ and $q$. Legend:\\
“GB”: All solutions exist globally and are bounded. \\
“IFTBU”: All solutions exist globally and some of them blow up in infinite time.\\
“FTBU”: There exist solutions blowing up in finite time. \\
“?”: There exist unbounded solutions but it seems to be open whether finite-time blow-up is possible.\\
shaded regions: added to the FTBU regime due to Theorem~\ref{th:ftbu}.}
\label{fig:mq_graph}
\end{minipage}

\paragraph{Blow-up induced by large negative energy.}
As discussed for instance in the survey \cite{LankeitWinklerFacingLowRegularity2019},
blow-up in fully parabolic chemotaxis systems is usually detected by exploiting the energy structure, i.e., the fact that
\begin{align}\label{eq:intro:F}
  \mc F(u, v) \defs \frac12 \intom |\nabla v|^2 + \frac12 \intom v^2 - \intom uv + \intom G(u),
  \quad \text{where} \quad G(u) \defs \int_1^u \int_1^\sigma \frac{\phi(\tau)}{\psi(\tau)} \dtau \dsigma,
\end{align}
decreases along trajectories of \eqref{prob:proto}.
Due to global temporal integrability of the dissipative term 
\begin{align*}
  \mathcal{D}(u,v) &\defs \into |\Delta v - v + u|^2 + \into \left| \frac{\phi(u)}{\sqrt{\psi(u)}} \nabla u - \sqrt{\psi(u)} \nabla v \right|^2,
\end{align*}
the $\omega$-limit sets of global bounded solutions each must contain a steady state of \eqref{prob:proto}.
However, in the supercritical case there is a minimal energy (depending on $\intom u$) below which no steady states exist,
meaning that solutions emanating from initial data with energy below this threshold must be unbounded (\cite{HorstmannWangBlowupChemotaxisModel2001}, \cite{WinklerDoesVolumefillingEffect2009}).
While this approach evidently does not provide any information on the finiteness of the blow-up time,
a finer analysis of the energy structure,
first performed in \cite{WinklerFinitetimeBlowupHigherdimensional2013} for $m = q = 1$
and then extended to quasilinear systems in \cite{CieslakStinnerFinitetimeBlowupGlobalintime2012, CieslakStinnerFiniteTimeBlowupSupercritical2014, CieslakStinnerNewCriticalExponents2015} and \cite{LaurencotMizoguchiFiniteTimeBlowup2017},
does: If $\F(u,v)\ge -C \left(\D^{\gamma}(u,v)+1\right)$ for some $C > 0$ and $\gamma \in (0, 1)$, then $-\mc F(u, v)$ solves a certain autonomous superlinear ODI,
so that solutions emanating from initial data with sufficiently large negative energy necessarily cease to exist after finite time.

\paragraph{Key new idea: Utilize pointwise upper estimates.}
In order to obtain such a relationship between $\mc F(u, v)$ and $\mc D(u, v)$, one splits $\Omega$ in an inner region $B_{r_0}$ and an outer region $\Omega \setminus B_{r_0}$.
For the analysis on the former, it turns out that one needs to (inter alia) control
\begin{align}\label{eq:intro:bad_term}
  \intb |x|^{2-\delta_0} u (u+1)^{1-q}
\end{align}
for an arbitrary $\delta_0 > 0$ by $\eta \into G(u) + C$ for sufficiently small $\eta > 0$ and arbitrary $C >0$.
If $q \ge 1$, this can be done by making use of conservation of mass, i.e., the fact that $\intom u$ is constant in time.
Likewise, for ($q < 1$ and) $m \ge 1$, \eqref{eq:intro:bad_term} can be estimated by $r_0^{2-\delta_0} \intom u (u+1)^{m-q}$
and hence by $\eta \intom G(u)$ if $r_0$ is sufficiently small.

This covers the case $\max\{m, q\} \ge 1$ already treated in \cite{CieslakStinnerFinitetimeBlowupGlobalintime2012, CieslakStinnerFiniteTimeBlowupSupercritical2014, CieslakStinnerNewCriticalExponents2015}.
For $\max\{m, q\} < 1$, we not only rely on mass conservation and the last summand of the functional $\mc F$ defined in \eqref{eq:intro:F},
but crucially also on the pointwise upper estimates of the form
\begin{align}\label{eq:intro:pw}
  u(r, t) \le C_\alpha r^{-\alpha}, \quad \alpha > \frac{n(n-1)}{n(m-q) + 1},
\end{align}
derived in \cite{FuestBlowupProfilesQuasilinear2020} (and earlier in \cite{WinklerBlowupProfilesLife2020} for the special case $m=q=1$).
That is, we do not estimate $|x|^{2-\delta_0}$ by $r_0^{2-\delta_0}$ but instead make full use of the dampening factor
and can hence control \eqref{eq:intro:bad_term} if \eqref{eq:ftbu:cond_main} holds, see Lemma~\ref{lem:u_g} below.

This approach also explains why the set of parameters $(m, q)$ known to admit solutions blowing up in finite time (cf.\ Figure~\ref{fig:mq_graph})
contains some points $(m, q')$, $(m, q'')$ with $q' < q''$ but not the line segment between them.
While larger $q$ should intuitively enhance the destabilizing effect of attractive chemotaxis
and hence finite-time blow-up should be possible for $(m, q)$ whenever it is for $(m, q')$ for some $q' < q$,
enlarging $q$ also means that the pointwise upper estimates provided by \cite{FuestBlowupProfilesQuasilinear2020},
on which our argument relies on, get substantially weaker.

To the best of our knowledge, Theorem~\ref{th:ftbu} is the first finite-time blow-up result for a \emph{fully parabolic} chemotaxis system
crucially making use of such upper pointwise estimates.
However, let us briefly point out that such an idea has been successfully employed for certain parabolic--elliptic relatives of \eqref{prob:proto},
where the second equation in \eqref{prob:proto} is either replaced by $0 = \Delta v - v + u$ 
or by $0 = \Delta v - \frac{1}{|\Omega|} \int_\Omega u + u$
(see, e.g., \cite{WinklerFinitetimeBlowupLowdimensional2018}, \cite{BlackEtAlRelaxedParameterConditions2021}, \cite{MizukamiEtAlCanChemotacticEffects2022}
and \cite{WinklerCriticalBlowupExponent2018}, \cite{FuestFinitetimeBlowupTwodimensional2020}, \cite{BlackEtAlRelaxedParameterConditions2021}, respectively).
While one obtains the same pointwise estimates for $u$ as for the fully parabolic system in the former case (also by relying on \cite{WinklerBlowupProfilesLife2020} or \cite{FuestBlowupProfilesQuasilinear2020}),
in the latter case, $u$ is radially decreasing whenever $u_0$ is (cf.\ \cite[Lemma~2.2]{WinklerCriticalBlowupExponent2018}),
which due to boundedness of mass implies the stronger bound $u(r, t) \le C r^{-n}$.
Furthermore, as observed and utilized first in \cite{FuestApproachingOptimalityBlowup2021} (and then in \cite{MarrasEtAlBehaviorTimeSolutions2023} and \cite{DuLiuBlowupSolutionsChemotaxis2023}, for instance),
in the latter situation even the slightly stronger bound  $u(r, t) \le r^{-n} \omega_n^{-1} n \int_{B_r(0)} u(\rho, t) \drho$ holds.

Moreover, the original main intention of estimates such as \eqref{eq:intro:pw}
has been to prove existence and properties of so-called blow-up profile functions (\cite{WinklerBlowupProfilesLife2020}, \cite{FuestBlowupProfilesQuasilinear2020}, \cite{fuest2022optimality}).
For a parabolic--elliptic counterpart of \eqref{prob:proto} and when $m = q = 1$, \cite{souplet2019blow} provides corresponding pointwise lower estimates as well, showing that the blow-up profile behaves like $|x|^{-2}$.
This in particular rules out a collapse into a Dirac-type singularity (as observed in 2D, see \cite{nagai2000chemotactic}, \cite{HerreroVelazquezBlowupMechanismChemotaxis1997}) in the higher-dimensional setting.

\paragraph{Plan of the paper.}
After stating concrete conditions on $\phi$ and $\psi$ in Section~\ref{sec:phi_psi},
a large part of our analysis focuses on deriving the estimate $\F(u,v) \ge -C (\D^{\gamma}(u,v)+1)$ for a certain class of functions $u, v$ in Section~\ref{sec:rel_f_d}.
In Section~\ref{sec:ftbu_init}, we then construct a sequence of initial data whose energy diverges to $-\infty$ but whose corresponding solutions all belong to a class studied in Section~\ref{sec:rel_f_d}.
To that end, we also revisit the proofs of \cite{WinklerBlowupProfilesLife2020} and \cite{FuestBlowupProfilesQuasilinear2020} in Appendix~\ref{sec:pw}
in order to ensure that the constant $C$ in the pointwise estimate in \eqref{eq:intro:pw} depends in a manageable way on the initial data.
Finally, we conclude Theorem~\ref{th:ftbu} in Section~\ref{sec:ftbu_final}.

\section{The functions \tops{$\phi$}{phi} and \tops{$\psi$}{psi}}\label{sec:phi_psi}
In this section, we present more general conditions on $\phi$ and $\psi$ which result in finite-time blow-up for appropriate initial data.
In addition to the inclusions $\phi, \psi \in C^2([0, \infty))$, we assume that 
\begin{align}\label{eq:cond:phi_psi_pos}
  \phi(s) > 0 
  \quad \text{and} \quad
  \frac{\psi(s)}{s} > 0
  \qquad \text{for all } s \ge 0,
\end{align}
and that there are $K_{\phi, 1}, K_{\phi, 2}, K_{\psi, 1}, K_{\psi, 2} > 0$,  such that
\begin{align}\label{eq:cond:phi_psi_mq}
 K_{\phi, 1} (s+1)^{m-1}\le 
  \phi(s) \le K_{\phi, 2} (s+1)^{m-1} \quad \text{and} \quad
 K_{\psi,1} s (s+1)^{q_1-1} \le \psi(s) \le K_{\psi,2} s(s+1)^{q_2-1}
\end{align}
for all $s \ge 0$,
with some $m \in \R$, $0 < q_1 \le q_2$ satisfying 
\begin{align}\label{eq:cond:m_q1q2_main}
  (\onemaxmq) \frac{n(n-1)}{(n(m-q_2)+1)_+} <2,
  \quad \text{where $\maxmq \defs \max\{q_1, m + q_1 - q_2\}$}.
\end{align}

Further conditions on $\phi$ and $\psi$ are expressed through the functions
\begin{alignat}{2}\label{eq:cond:G_def}
  G &\colon [0, \infty) \to [0, \infty), &&\quad s \mapsto \int_{s_0}^s \int_{s_0}^\sigma \frac{\phi(\tau)}{\psi(\tau)} \dtau \dsigma \quad \text{and} \\
  \label{eq:cond:H_def}
  H &\colon [0, \infty) \to \R, &&\quad s \mapsto \int_{s_0}^s \frac{\sigma\phi(\sigma)}{\psi(\sigma)} \dsigma,
\end{alignat}
where $s_0 > 1$ is fixed.
(Both $G(0)$ and $H(0)$ are, indeed, finite thanks to \eqref{eq:cond:phi_psi_mq}.)
If $n=2$, we assume that 
\begin{align}\label{eq:cond:G}
  \exists a > 0\; \exists \theta \in (0, 1)\; \forall s \ge s_0:
  G(s) &\le a s \ln^\theta s \quad \text{and} \\
\label{eq:cond:H}
  \exists b > 0 \; \forall s \ge s_0:
  H(s) &\le \frac{bs}{\ln s},
\intertext{while for $n \ge 3$ we require that}\label{eq:cond:G_higher}
	\exists a > 0\; \exists \theta \in \left(\frac2n, 1\right) \forall s \ge s_0:
  G(s) &\le a s^{2-\theta} \quad \text{and} \\
\label{eq:cond:H_higher}
  \exists \vt \in (0, 1), K > 0 \; \forall s \ge s_0:
  H(s) &\le \frac{n-2-\vt }{n} \cdot G(s)+K(s+1).
\end{align}
We note that a direct calculation based on the lower bound for $\phi$ and the upper bound for $\psi$ in \eqref{eq:cond:phi_psi_mq}
as well as \eqref{eq:cond:G} (if $n=2$) or \eqref{eq:cond:G_higher} (if $n\ge3$) shows
\begin{align}\label{eq:cond:m_q2}
	m-q_2 < \frac{n-2}{n}.
\end{align}
On the other hand, the following (well-known) examples inter alia show that whenever $m \in \R$ and $q = q_1 = q_2 > 0$ fulfill \eqref{eq:cond:m_q1q2_main} and \eqref{eq:cond:m_q2},
then indeed functions $\phi$, $\psi$ with the above properties exist.
\begin{lemma}\label{lem:phi_psi_admissible}
  Let $n \ge 2$, $R > 0$, $\Omega \defs B_R(0) \subset \R^n$.
  For
	\begin{enumerate}
    \item 
      $\phi(s) = (s+1)^{m-1}$, $\psi(s) = s(s+1)^{q-1}$, $s \ge 0$, with $m \in \R$, $q > 0$ satisfying \eqref{eq:ftbu:cond_unbdd} and \eqref{eq:ftbu:cond_main},

    \item
      $n = 2$ and $\phi(s) = (s+1)^{m-1}$, $\psi(s) = s(s+1)^{q-1} \ln(s+\ure)$, $s \ge 0$, with $m = q > 0$,
  \end{enumerate}
  we can find $K_{\phi, i}, K_{\psi, i} > 0$, $0 < q_1 \le q_2$ satisfying \eqref{eq:cond:phi_psi_mq} and \eqref{eq:cond:m_q1q2_main}
  as well as \eqref{eq:cond:G} and \eqref{eq:cond:H} if $n = 2$ and \eqref{eq:cond:G_higher} and \eqref{eq:cond:H_higher} if $n \ge 3$.
\end{lemma}
\begin{proof}
  \begin{enumerate}
    \item
      Setting $q_1 \defs q_2 \defs q > 0$, we see that \eqref{eq:cond:phi_psi_mq} holds with $K_{\phi, i} = K_{\psi, i} = 1$,
      and that \eqref{eq:cond:m_q1q2_main} is equivalent to \eqref{eq:ftbu:cond_main}.
      Since $\frac{\psi(s)}{\phi(s)} = s(s+1)^{q-m}$ can be estimated from below by $c_1 s \ln s$ if $n = 2$ (since $q > m$) or by $c_1 s^\alpha$ with $\alpha = 1+q-m > \frac{2}{n}$ if $n \ge 3$
      for all $s \ge s_0$ and some $s_0 > 1$ and $c_1 > 0$,
      the proof of \cite[Corollary~5.2~(i) and (iii)]{WinklerDoesVolumefillingEffect2009} yields \eqref{eq:cond:G}--\eqref{eq:cond:H_higher}.

    \item
      We let $q_1 \defs q$ and fix $q_2\in(q,\frac32 q)$.
      Then $\ln(s+\ure) \le K_{\psi, 2} (s+1)^{q_2-q}$ for all $s \ge 0$ and some $K_{\psi, 2} > 0$
      and \eqref{eq:cond:phi_psi_mq} holds with $K_{\phi, i} = K_{\psi, 1} = 1$.
      If $q < 1$, $(1-\max\{q, q+q-q_2\})\frac{2}{2(q-q_2)+1} < (1-q) \frac{2}{-q+1} = 2$, and hence \eqref{eq:cond:m_q1q2_main} holds for all choices of $q$.
      As to \eqref{eq:cond:G} and \eqref{eq:cond:H},
      we again make use of the criterion in \cite[Corollary~5.2~(i)]{WinklerDoesVolumefillingEffect2009},
      which is applicable since $\frac{\psi(s)}{\phi(s)} = s\ln(s + \ure) \ge s \ln s$ for all $s > 0$.
     \qedhere
  \end{enumerate}
\end{proof}

Under these assumptions, there exist solutions of \eqref{prob:proto} blowing up in finite time.
That is, in the following sections we shall prove the following generalization of Theorem~\ref{th:ftbu}.
\begin{theorem}\label{th:ftbu:q1q2}
  Let $n \ge 2$, $R > 0$, $\Omega \defs B_R(0) \subset \R^n$,
  $\phi, \psi \in C^2([0, \infty)$ be such that \eqref{eq:cond:phi_psi_pos} holds,
  $K_{\phi, i}, K_{\psi, i} > 0,$, $m \in \R$, $0 < q_1 \le q_2$ satisfy \eqref{eq:cond:phi_psi_mq}--\eqref{eq:cond:m_q1q2_main}
  and suppose that \eqref{eq:cond:G} and \eqref{eq:cond:H} hold if $n = 2$ and that \eqref{eq:cond:G_higher} and \eqref{eq:cond:H_higher} hold if $n \ge 3$.
  For each $M_u > 0$, there then exist nonnegative, radially symmetric initial data $u_0, v_0 \in \con\infty$ with $\intom u = M_u$
  and a uniquely determined classical solution $(u, v)$ of \eqref{prob:proto}
  which blows up at some finite time $\tmax$ in the sense that $\limsup_{t \nea \tmax} \|u(\cdot, t)\|_{\leb\infty} = \infty$.
\end{theorem}

\section{Relating \tops{$\mc F(u, v)$}{F(u, v)} and \tops{$\mc D(u, v)$}{D(u, v)}}\label{sec:rel_f_d}
In this section, we prove a relation between the energy functional $\mc F$ and its dissipation rate $\mc D$
which implies that $-\mc F(u, v)$ solves a superlinear ODE when $(u, v)$ solves \eqref{prob:proto}.
In order to eventually conclude from this that $\mc F(u, v)$, and hence $(u, v)$, cannot exist globally for certain initial data,
we need to carefully track the dependence of the constants below.
To that end, the present section is only concerned with functions $(u, v)$ belonging to a set $\mc S$ defined in \eqref{eq:def_S} below.
In Section~\ref{sec:ftbu_init}, we will then construct a family of suitable initial data whose corresponding solutions of \eqref{prob:proto} belong to $\mc S$ throughout evolution.

More precisely, we first let
\begin{align}\label{eq:ass_fd}
  \begin{cases}
    n \ge 2,\, R > 0,\, \Omega \defs B_R(0) \subset \R^n, \\
    \phi, \psi \in C^2([0, \infty) \text{ be such that \eqref{eq:cond:phi_psi_pos} holds}, \\
    K_{\phi, i}, K_{\psi, i} > 0,\, m \in \R,\, 0 < q_1 \le q_2 \text{ and } s_0 > 1 \text{ such that} \\
    \text{\eqref{eq:cond:phi_psi_mq} as well as \eqref{eq:cond:H} (if $n = 2$) or \eqref{eq:cond:H_higher} (if $n \ge 3$) hold,} \\
    \maxmq \text{ as in \eqref{eq:cond:m_q1q2_main} and } r_\star \defs \min\{\frac1\ure, \frac{R}{2}\}
  \end{cases}
\end{align}
(but do not require \eqref{eq:cond:m_q1q2_main} nor \eqref{eq:cond:G_def} nor \eqref{eq:cond:G} to hold in this section)
and then fix
\begin{align}\label{eq:params_for_S}
  M_u, M_v, A, B > 0, \quad \alpha \in \left(0, \frac{2}{(\onemaxmq)_+} \right) \quad \text{and} \quad \kappa \ge n.
\end{align}
We shall focus on radially symmetric functions $(u, v)$ not necessarily solving \eqref{prob:proto} but instead having the properties
\begin{align}
\label{eq:l1_bdd}
 & \into u = M_u, \quad \into v\le M_v,\\[0.5em]
 \label{eq:upper_u}
 & \text{if $\maxmq < 1$, then } u(r) + 1 \le A r^{-\alpha} \text{ for all $r \in (0, R)$}, \\[0.5em]
 \label{eq:upper_v}
 & v(r)\le B r^{-\kappa}\text{ for all $r \in (0, R)$}.
\end{align}
(Here and below, we write $\varphi(|x|)$ for $\varphi(x)$ for radially symmetric functions $\varphi$.)
Thus, we define the spaces
\begin{align}\label{eq:def_S}\nn
  \mc S_u &\defs \mathcal{S}_u(M_u, A, \alpha) \defs \Big\{\,
    u \in C_{\mathrm{rad}, \mathrm{pos}}^2(\overline \Omega) \;\Big|\;
    u \text{ satisfies } \eqref{eq:l1_bdd} \text{ and } \eqref{eq:upper_u}\,\Big\}, \\
  \mc S_v &\defs \mathcal{S}_v(M_v, B, \kappa) \defs \Big\{\,
    v \in C_{\mathrm{rad}, \mathrm{pos}}^2(\overline \Omega) \;\Big|\;
    v \text{ satisfies } \partial_\nu v = 0 \text{ on } \partial \Omega,\, \eqref{eq:l1_bdd} \text{ and } \eqref{eq:upper_v}\,\Big\}, \nn\\
  \mc S &\defs \mathcal{S}(M_u, M_v, A, B, \alpha, \kappa) \defs \mathcal{S}_u(M_u, A, \alpha) \times \mathcal{S}_v(M_v, B, \kappa),
\end{align}
where $C_{\mathrm{rad}, \mathrm{pos}}^2(\overline \Omega) \defs \{\,\varphi \in \con2 \mid \varphi \text{ is nonnegative and radially symmetric}\,\}$.
Finally, we set 
\begin{align}\label{eq:def_F}
	\mathcal{F}(u,v) &\defs \frac12 \intom |\nabla v|^2 + \frac12 \intom v^2 - \intom uv + \intom G(u) \quad \text{and} \\
  \mathcal{D}(u,v) &\defs \into f^2+\into g^2 \label{eq:def_D}
\end{align}
for $(u, v) \in \mc S$, where
\begin{align}\label{eq:f}
 f &\defs \Delta v-v+u \quad \text{and} \\\label{eq:g}
 g &\defs \left( \frac{\phi(u)}{\sqrt{\psi(u)}} \nabla u - \sqrt{\psi(u)} \nabla v \right) \cdot \frac{x}{|x|}.
\end{align}

The goal of the present section is to prove the following.
\begin{lemma}\label{lm:F_ge_C_theta}
  Assume \eqref{eq:ass_fd} and \eqref{eq:params_for_S}.
  Then there are $C>0$ and $\gamma \in (0, 1)$ such that for all $(u, v) \in \mc S$,
  \begin{align}\label{eq:F_ge_C_theta:est}
    \F(u,v)\ge -C \left(\D^{\gamma}(u,v)+1\right).
  \end{align}
\end{lemma}
We emphasize that $C$ may depend on the quantities in \eqref{eq:ass_fd} and \eqref{eq:params_for_S} only.

Plugging \eqref{eq:f} into \eqref{eq:def_F} and integrating by parts shows $\mc F(u, v) = -\frac12 \intom |\nabla v|^2 - \frac12 \intom v^2 - \intom v f + \intom G(u)$,
which suggests that upper estimates for $\intom |\nabla v|^2$ are essential for showing \eqref{eq:F_ge_C_theta:est}.
To that end, we follow \cite{WinklerFinitetimeBlowupHigherdimensional2013} and first focus on suitable estimates holding in the inner region $B_{r_0} \defs B_{r_0}(0)$.

\begin{lemma}\label{lem:nabla_c}
  Assume \eqref{eq:ass_fd} and \eqref{eq:params_for_S} and let
  \begin{align}\label{eq:def_ell}
    \ell(r) \defs
    \begin{cases}
      - \ln r, & n = 2, \\[0.5em]
      \dfrac{1}{n-2}, & n \ge 3,
    \end{cases}
    \qquad \text{for $r > 0$}.
  \end{align}
  Then there exists $C > 0$ such that for all $(u, v) \in \mc S$, all $r_0\in(0, r_\star]$ and all $\mu \in (0, 1)$, 
	\begin{align}\label{eq:lem:nabla_c}
	\nn
	      \intb |\nabla v|^2
  &\le  \frac{2n \omega_n}{1-\mu} \int_0^{r_0} r^{n-1} \ell(r) (H(u))_+
        - \frac{1}{1-\mu} \intb v^2
        + \frac{C r_0}{\mu(1-\mu)} \intb f^2 \\
	&\pe  +\, \frac{2 \omega_n}{1-\mu} \intbr r^n \ell(r) \frac{u}{\sqrt{\psi(u)}} g \dr
        + \frac{\omega_n}{1-\mu} r_0^n \ell(r_0) \big(v^2(r_0) + 2|H(0)| - v_r^2(r_0)\big).
		\end{align}
\end{lemma}
\begin{proof}
  We fix $(u, v) \in \mc S$, $r_0 \in (0, r_\star)$ and $\mu \in (0, 1)$.
	Upon integrating by parts and by \eqref{eq:f},
	\begin{align}\label{eq:lem:nabla_c:2}
	\nn
	 \int_0^{r_0} r^{n-1} v_r^2 \dr
		&= \intbr (r^{n-1} v_r)^2 (-r^{2-n} \ell(r))_r \dr \\
		\nn
		&= -(r_0^{n-1}v_r(r_0))^2 r_0^{2-n} \ell(r_0) + \intbr r^{2-n} \ell(r) \cdot 2r^{n-1}v_r\cdot(r^{n-1} v_r)_r \dr\\
		\nn
		&= -r_0^{n} \ell(r_0) v_r^2(r_0) + 2\intbr r \ell(r) v_r(r^{n-1} f +r^{n-1} v-r^{n-1} u) \dr\\
		&= -r_0^{n} \ell(r_0) v_r^2(r_0) + 2\intbr r^n \ell(r) v_r f \dr + \intbr r^n \ell(r) (v^2)_r \dr - 2 \intbr r^n \ell(r) uv_r \dr.
	\end{align}
  Here, Young's inequality warrants that
  \begin{align}\label{eq:lem:nabla_c:3}
     2\intbr r^n \ell(r) v_r f \dr &\le \mu \intbr r^{n-1} v_r^2 \dr + \frac{1}{\mu} \intbr r^{n+1} \ell^2(r) f^2 \dr
  \end{align}	
  and another integration by parts shows
  \begin{align}\label{eq:lem:nabla_c:4}
        \intbr r^n \ell(r) (v^2)_r \dr
     &= r_0^n \ell(r_0) v^2(r_0) - \intbr \big(nr^{n-1} \ell(r) + r^n \ell'(r)\big) v^2 \dr \nn \\
     &\le r_0^n \ell(r_0) v^2(r_0) - \intbr r^{n-1} v^2 \dr,
  \end{align}
  since $n \ell(r) + r \ell'(r) \ge 1$ for all $r \in (0, r_0) \subset (0, \frac1\ure)$.
  Recalling (\ref{eq:g}) and (\ref{eq:cond:H_def}), we see that
  \begin{align}\label{eq:lem:nabla_c:5}
  \nn
    - 2 \intbr r^n \ell(r) uv_r  \dr &= - 2 \intbr r^n \ell(r) \left(\frac{u\phi(u)}{\psi(u)}u_r-\frac{u}{\sqrt{\psi(u)}} g \right ) \dr\\
    &= - 2 \intbr r^n \ell(r) (H(u))_r \dr + 2 \intbr r^n \ell(r) \frac{u}{\sqrt{\psi(u)}} g \dr,
  \end{align}
  where a final integration by parts gives
  \begin{align}\label{eq:lem:nabla_c:6}
          - 2\intbr r^n \ell(r) (H(u))_r \dr \nn
    &=    - 2r_0^n \ell(r_0) H(u(r_0))
          + 2n \intbr r^{n-1} \ell(r) H(u) \dr
          + 2\intbr r^{n} \ell'(r) H(u) \dr \\
    &\le  2r_0^n \ell(r_0) |H(0)|
          + 2n \intbr r^{n-1} \ell(r) (H(u))_+ \dr,
  \end{align}	
  since $\ell'(r) H(u) \le r^{-1} \ell(r) (H(u))_-$ for $r \in (0, r_0)$.
  Collecting \eqref{eq:lem:nabla_c:2}--\eqref{eq:lem:nabla_c:6} yields
  \begin{align*}
          (1-\mu)\intbr r^{n-1}v_r^2 \dr
    &\le  2n \intbr r^{n-1} \ell(r) (H(u))_+ \dr
          - \intbr r^{n-1} v^2 \dr
          + \frac{r_0}{\mu} \intbr r^{n} \ell^2(r) f^2 \dr \\
    &\pe  +\, 2 \intbr r^n \ell(r) \frac{u}{\sqrt{\psi(u)}} g \dr
          + r_0^n \ell(r_0) \big(v^2(r_0) + 2|H(0)| - v_r^2(r_0)\big),
  \end{align*}
  which implies \eqref{eq:lem:nabla_c} due to $\intb |\nabla v|^2 = \omega_n \int_0^{r_0} r^{n-1} v_r^2 \dr$
  upon setting $C \defs \sup_{r \in (0, r_\star)} r \ell^2(r)$.
\end{proof}

Estimating the fourth term on the right-hand side in \eqref{eq:lem:nabla_c} crucially relies on the pointwise upper estimate in \eqref{eq:upper_u}.
That is, the following lemma constitutes the decisive difference to the analysis performed in the precedents \cite{CieslakStinnerFinitetimeBlowupGlobalintime2012,CieslakStinnerFiniteTimeBlowupSupercritical2014,CieslakStinnerNewCriticalExponents2015}.
\begin{lemma}\label{lem:u_g}
  Assume \eqref{eq:ass_fd} and \eqref{eq:params_for_S}, let $\ell$ be as in \eqref{eq:def_ell} and let $\eta_1 > 0$.
	We can find $\delta_0>0$ and $C>0$ such that for all $u \in \mc S_u$ and all $r_0\in(0, r_\star]$, 
	\begin{align}\label{eq:lem:u_g}
		\omega_n \intbr r^n \ell(r) \frac{u}{\sqrt{\psi(u)}} g \dr\le \eta_1 \intb G(u) + C r_0^{\delta_0} \norm[L^2(\Omega)]{g}^2 + C.
	\end{align}
\end{lemma}
\begin{proof}
  Since $\alpha < \frac{2}{(\onemaxmq)_+}$ by \eqref{eq:params_for_S},
  there exists $\delta_1 > 0$ such that
  \begin{align}\label{eq:u_g:delta_1_def}
    \delta_1 + \alpha (\onemaxmq)_+ = 2.
  \end{align}
  We moreover fix $\delta_0 \in (0, \delta_1)$
  and then $c_1 > 0$ such that
  \begin{align}\label{eq:u_g:delta_0_def}
    r^{-\delta_0} \ell^2(r) \le c_1 r^{-\delta_1} 
    \qquad \text{for all $r \in (0, R)$}.
  \end{align}
  If $m > q_2$, \eqref{eq:cond:phi_psi_mq} implies
  \begin{align*}
          \frac{K_{\psi, 2}}{K_{\phi, 1}} G(s)
    &\ge  \int_{s_0}^s \int_{s_0}^\sigma (\tau + 1)^{m-q_2-1} \dtau \dsigma  \\
    &=    \frac{(s + 1)^{m-q_2+1}}{(m-q_2+1)(m-q_2)} - \frac{(s_0 + 1)^{m-q_2+1}}{(m-q_2+1)(m-q_2)} - \frac{(s-s_0)(s_0+1)^{m-q_2}}{m-q_2}
  \end{align*}
  for all $s \ge 0$,
  so that, independently of the sign of $m-q_2$, we may fix $c_2 > 0$ with
  \begin{align}\label{eq:u_g:G_est}
    s(s+1)^{(m-q_2)_+} \le c_2 G(s) + c_2 s + c_2
    \qquad \text{for all $s \ge 0$}.
  \end{align}
  Using Young's inequality, \eqref{eq:cond:phi_psi_mq} and \eqref{eq:u_g:delta_0_def}, we see that there is $c_3 > 0$ such that
  \begin{align}\label{eq:u_g:est_young}
    &\pe  \omega_n \intbr r^n \ell(r) \frac{u}{\sqrt{\psi(u)}} g \dr \nn \\
    &\le  c_3 \omega_n \int_0^{r_0} r^{n-1+\delta_0} g^2 \dr 
          + \frac{\eta_1 K_{\psi, 1} \omega_n}{c_1 c_2 \max\{1, A^{\onemaxmq}\}} \intbr r^{n+1-\delta_0} \ell^2(r) \frac{u^2}{\psi(u)} \dr \nn \\
    &\le  c_3 r_0^{\delta_0} \|g\|_{L^2(\Omega)}^2
          + \frac{\eta_1 \omega_n}{c_2 \max\{1, A^{\onemaxmq}\}}  \intbr r^{n+1-\delta_1} u(u+1)^{1-q_1} \dr
  \end{align}
  for all $u \in \mc S_u$ and all $r_0\in(0, r_\star]$.
  If $\maxmq = \max\{q_1, m+q_1-q_2\} \ge 1$, we  have
   \begin{align*}
          \intbr r^{n+1-\delta_1} u(u+1)^{1-q_1} \dr
    &\le  \intbr r^{n+1-\delta_1} u(u+1)^{(m-q_2)_+} \dr
  \end{align*}
  for all $u \in \mc S_u$ and all $r_0\in(0, r_\star]$.
  If on the other hand $\maxmq < 1$, then \eqref{eq:upper_u} and \eqref{eq:u_g:delta_1_def} warrant that
  \begin{align*}
          \intbr r^{n+1-\delta_1} u(u+1)^{1-q_1} \dr
    &\le  A^{\onemaxmq} \intbr r^{n+1-\delta_1- \alpha (\onemaxmq)} u(u+1)^{(m-q_2)_+} \dr \\
    &=    A^{\onemaxmq} \intbr r^{n-1} u(u+1)^{(m-q_2)_+} \dr
  \end{align*}
  for all $u \in \mc S_u$ and all $r_0\in(0, r_\star]$.
  Due to \eqref{eq:u_g:G_est} and \eqref{eq:l1_bdd}, we conclude that in both situations,
  \begin{align*}
    &\pe  \frac{1}{\max\{1, A^{\onemaxmq}\}} \intbr r^{n+1-\delta_1} u(u+1)^{1-q_1} \dr \\
    &\le  \intbr r^{n-1} u(u+1)^{(m-q_2)_+} \dr
     \le  c_2 \intbr r^{n-1} G(u) \dr
          + c_2 \intr r^{n-1} u \dr
          + \frac{c_2 R^n}{n} \\
    &=    \frac{c_2}{\omega_n} \int_{B_{r_0}} G(u)
          + \frac{c_2 M_u}{\omega_n}
          + \frac{c_2 R^n}{n}
    \qquad \text{for all $u \in \mc S_u$ and all $r_0 \in (0, r_\star)$}.
  \end{align*}
  Together with \eqref{eq:u_g:est_young}, this yields \eqref{eq:lem:u_g}.
\end{proof}

Next, we deal with the first term in \eqref{eq:lem:nabla_c}, similarly as in \cite{CieslakStinnerFiniteTimeBlowupSupercritical2014} (if $n = 2$) or \cite{CieslakStinnerFinitetimeBlowupGlobalintime2012} (if $n \ge 3$).
\begin{lemma}\label{lem:hu}
  Assume \eqref{eq:ass_fd} and \eqref{eq:params_for_S} and let $\ell$ be as in \eqref{eq:def_ell}.
	Then there are $\eta_2 \in (0, 1)$ and $C>0$ such that for all $u \in \mc S_u$ and all $r_0\in(0, r_\star]$,
	\begin{align}\label{eq:lem:hu}
        n \omega_n \intbr r^{n-1} \ell(r) (H(u))_+ \dr
    \le (1-\eta_2) \intb G(u) + C.
	\end{align}
\end{lemma}
\begin{proof}
  We first suppose that $n = 2$.
  Then \eqref{eq:cond:H_def} and \eqref{eq:cond:H} imply $H(s) \le 0$ for $s < s_0$ as well as
  \begin{align*}
        H(s) \ln H(s)
    \le \frac{bs}{\ln s} \ln\left( \frac{bs}{\ln s} \right)
    =   \frac{bs}{\ln s} \left( \ln b + \ln s - \ln \ln s \right)
    \le \frac{b(|\ln b| + 1 + |\ln \ln s_0|)}{\min\{1, \ln s_0\}} \cdot s
  \end{align*}
  for all $s \ge s_0 > 1$.
  Since moreover $\ln r^{-1} \cdot H(s) \le r^{-1} + H(s) \ln H(s) - H(s)$ for all $r, s > 0$ by the Fenchel--Young inequality,
  we conclude that there is $c_1 > 0$ such that
  \begin{align*}
          n \intbr r^{n-1} \ell(r) (H(u))_+ \dr
    &\le  n \intbr 1 \dr + n \intbr r^{n-1} \mathds 1_{\{u \ge s_0\}} H(u) \ln H(u) \dr \nn \\
    &\le  n r_0  + c_1 \int_0^R r^{n-1} u \dr
     \le  nR + \frac{c_1 M_u}{\omega_n}
  \end{align*}
  for all $u \in \mc S_u$ and all $r_0\in(0, r_\star]$.
   
  If $n \ge 3$ and hence $\ell(r) = \frac{1}{n-2}$, \eqref{eq:cond:H_higher} asserts that there is $\eta_2 \in (0, 1)$ such that
  \begin{align*}
          n \intbr r^{n-1} \ell(r) (H(u))_+ \dr
    &\le  \frac{n-2-\vt}{n-2} \intbr r^{n-1} G(u) \dr + \frac{nK}{n-2} \int_0^R r^{n-1} (u + 1) \dr \nn \\
    &\le  (1-\eta_2) \intbr r^{n-1} G(u) \dr + \frac{nK}{n-2} \left(\frac{M_u}{\omega_n} + \frac{R^n}{n}\right)
  \end{align*}
  for all $u \in \mc S_u$ and all $r_0\in(0, r_\star]$,
  so that in both cases \eqref{eq:lem:hu} holds for some $\eta_2 \in (0, 1)$ and $C > 0$.
\end{proof}

Combining Lemmata~\ref{lem:nabla_c}--\ref{lem:hu} now provides the following estimate for $\F(u, v)$ in the inner region.
\begin{lemma}\label{lem:F_inner}
  Assume \eqref{eq:ass_fd} and \eqref{eq:params_for_S}.
  There exists $C>0$ such that for all $(u, v) \in \mc S$ and all $r_0\in(0, r_\star]$, we have
  \begin{align}\label{eq:lem:F_inner_higher}
    &\pe \intb uv-\frac 12\intb |\nabla v|^2-\frac 12 \intb v^2-\intb G(u) \nn \\
    &\le C \norm[L^2(\Om)]{f}^{\frac{2n+4}{n+2}}+C r_0^{\delta_0}\norm[L^2(\Om)]{f}^2 +C r_0^{\delta_0}\norm[L^2(\Om)]{g}^2 +C r_0^{n-2-2\kappa}.
  \end{align}
\end{lemma}
\begin{proof}
  We let $\eta_2 \in (0, 1)$ be as given by Lemma~\ref{lem:hu} and choose $\eta_1 \in (0, \eta_2)$, $\mu \in (0, 1)$ and $\eps_0 \in (0, 1)$ so small that
  \begin{align}\label{eq:cond_eps0}
    \frac{(1 + \eps_0)(1 - \eta_2 + \eta_1)}{1-\mu}  \le 1.
  \end{align}
  According to the Gagliardo--Nirenberg inequality, there is $c_1 > 0$ such that
  \begin{align}\label{eq:gni}
        \|\varphi\|_{L^2(B_{r_0})}^\frac{2n+4}{n}
    \le c_1 \|\nabla \varphi\|_{L^2(B_{r_0})}^2 \|\varphi\|_{L^1(B_{r_0})}^\frac4n + c_1 r_0^{-n-2} \|\varphi\|_{L^1(B_{r_0})}^\frac{2n+4}{n}
    \qquad \text{for all $\varphi \in W^{1, 2}(B_{r_0})$ and all $r_0 > 0$}.
  \end{align}
  (A scaling argument shows that $c_1$ does, indeed, not depend on $r_0$.)
  By \eqref{eq:f}, an integration by parts, Hölder's inequality, Young's inequality, \eqref{eq:gni} and \eqref{eq:l1_bdd}, we obtain
  \begin{align*}
    \intb uv
    &= \intb fv+ \intb |\nabla v|^2 +\intb v^2-\int_{\pa B_{r_0}} \pa_{\nu} v \cdot v\\\nn
    &\le c_2 \|f\|_{L^2(B_{r_0})}^{\frac{2n+4}{n+4}} + \frac{\eps_0}{2c_1 M_v^{4/n}} \|v\|_{L^2(B_{r_0})}^\frac{2n+4}{n} + \intb |\nabla v|^2+\intb v^2- \omega_n r_0^{n-1}v_r(r_0) v(r_0)\\
    &\le \left(1+ \frac{\eps_0}{2}\right) \intb |\nabla v|^2 +\intb v^2+c_2\norm[L^2(\Omega)]{f}^{\frac{2n+4}{n+4}}+\frac{\eps_0M_v^2 r_0^{-n-2}}{2}-\omega_n r_0^{n-1}v_r(r_0) v(r_0)
  \end{align*}
  for some $c_2 > 0$, all $(u, v) \in \mc S$ and all $r_0\in(0, r_\star]$.
  Together with Lemma~\ref{lem:nabla_c}, Lemma~\ref{lem:u_g}, Lemma~\ref{lem:hu}, \eqref{eq:cond_eps0}, Young's inequality, \eqref{eq:upper_v} and \eqref{eq:params_for_S}, this yields
  \begin{align}
    &\pe \intb uv-\frac 12\intb |\nabla v|^2-\frac 12 \intb v^2-\intb G(u) \nn\\[0.3em]
    &\le \frac{1+\eps_0}{2} \intb |\nabla v|^2 +\frac 12\intb v^2
    +c_2\norm[L^2(\Omega)]{f}^{\frac{2n+4}{n+4}}-\intb G(u)+M_v^2 r_0^{-n-2}-\omega_n r_0^{n-1}v_r(r_0) v(r_0) \nn\\[0.3em]
    &\le  \frac{(1+\eps_0)n\omega_n}{1-\mu} \intbr r^{n-1} \ell(r) (H(u))_+ \dr
          + \frac12 \left(1 - \frac{1+\eps_0}{1-\mu} \right) \intb v^2 \nn\\[0.3em]
    &\pe  +\, c_2 \|f\|_{\leb2}^\frac{2n+4}{n+4}
          + \frac{c_3(1+\eps_0) r_0}{2\mu(1-\mu)} \|f\|_{\leb2}^2
          + \frac{(1+\eps_0)\omega_n}{1-\mu} \intbr r^n \ell(r) \frac{u}{\sqrt{\psi(u)}} g \dr
          - \intb G(u) \nn\\[0.3em]
    &\pe  +\, \frac{(1+\eps_0)\omega_n}{1-\mu} r_0^n \ell(r_0) \big(v^2(r_0) + 2|H(0)| - v_r^2(r_0)\big)
          + M_v^2 r_0^{-n-2}
          - \omega_n r_0^{n-1}v_r(r_0) v(r_0) \nn \\[0.3em]
    &\le  \left( \frac{(1 + \eps_0)(1 - \eta_2 + \eta_1)}{1-\mu} - 1 \right) \intb G(u)
          + c_2 \|f\|_{\leb2}^\frac{2n+4}{n+4}
          + \frac{\hat c_3 r_0}{2\mu} \|f\|_{\leb2}^2
          + \hat c_4 r_0^{\delta_0} \|g\|_{\leb2}^2 \nn\\[0.3em]
    &\pe  +\, \frac{(1+\eps_0)\omega_n}{1-\mu} r_0^n v_r^2(r_0) \big(\ell(r_\star) - \ell(r_0)\big)
          + \left(c_6 + \frac{\omega_n(1-\mu)}{4\ell(r_\star)(1+\eps_0)}\right) r_0^{n-2} v^2(r_0) \nn\\[0.3em]
    &\pe  +\, \hat c_4
          + \hat c_5
          + 2c_6r_0^{n-2} |H(0)|
          + M_v^2 r_0^{-n-2} \nn\\[0.3em]
    &\le  c_2 \|f\|_{\leb2}^\frac{2n+4}{n+4}
          + \frac{\hat c_3}{2\mu} r_0^{\delta_0} \|f\|_{\leb2}^2
          + \hat c_4 r_0^{\delta_0} \|g\|_{\leb2}^2 \nn\\[0.3em]
    &\pe  +\, \left[ B^2 \left(c_6 + \frac{\omega_n(1-\mu)}{4\ell(r_\star)(1+\eps_0)}\right) 
            + \hat c_4
            + \hat c_5
            + 2c_6r_\star^{n-2} |H(0)|
            + M_v^2
          \right] r_0^{n-2-2\kappa}
  \end{align}
  for all $(u, v) \in S$ and all $r_0\in(0, r_\star]$,
  where $c_3, c_4, c_5$ are the constants given by Lemma~\ref{lem:nabla_c}, Lemma~\ref{lem:u_g} and Lemma~\ref{lem:hu}, respectively,
  $\hat c_i \defs \frac{(1+\eps_0)c_i}{1-\mu}$ for $i \in \{3, 4, 5\}$,
  $\ell$ is as in \eqref{eq:def_ell},
  $\delta_0$ is as given by Lemma~\ref{lem:u_g},
  and $c_6 \defs \sup_{r \in (0, r_\star)} \frac{(1+\eps_0)\omega_n}{1-\mu}r^2 \ell(r)$.
  Thus, \eqref{eq:lem:F_inner_higher} holds for an evident choice of $C$.
\end{proof}

In contrast to the inner region treated above,
suitably estimating $\mathcal{F}(u, v)$ in the outer region $\Omega \setminus B_{r_0}$ can be achieved by just relying on \eqref{eq:l1_bdd} and \eqref{eq:upper_v}.
\begin{lemma}\label{lem:F_outer}
  Assume \eqref{eq:ass_fd} and \eqref{eq:params_for_S}.
  Then 
	\begin{align}
		\int_{\Omega \setminus B_{r_0}} uv-\frac 12\int_{\Omega \setminus B_{r_0}} |\nabla v|^2-\frac 12 \int_{\Omega \setminus B_{r_0}} v^2-\int_{\Omega \setminus B_{r_0}} G(u)
    \le BM_u r_0^{-\kappa}
	\end{align}
  for all $(u, v) \in \mc S$ and all $r_0\in(0, R)$.
\end{lemma}
\begin{proof}
  By dropping some nonpositive terms and applying \eqref{eq:upper_v} and \eqref{eq:l1_bdd}, we see that
	\begin{align*}
	    \int_{\Omega\backslash B_{r_0}} \left( uv-\frac 12 |\nabla v|^2-\frac 12 v^2- G(u) \right)
	 \le \sup_{r \in (r_0, R)} v(r) \cdot \int_{\Omega\backslash B_{r_0}} u
	\le B\sup_{r \in (r_0, R)} r^{-\kappa} \cdot \into u
  =   BM_u r_0^{-\kappa}
	\end{align*}
  for all $(u, v) \in \mc S$ and all $r_0 \in (0, R)$.
\end{proof}

Finally, by combining Lemma~\ref{lem:F_inner} and \ref{lem:F_outer} for appropriately chosen $r_0$ we are able to prove Lemma~\ref{lm:F_ge_C_theta},
that is, the estimate $\mc F(u, v) \ge -C (\mc D^\gamma(u, v) + 1)$ for some $C > 0$ and $\gamma \in (0, 1)$.
\begin{proof}[Proof of Lemma~\ref{lm:F_ge_C_theta}]
  We let $r_0 \defs \min\left\{ r_\star, \big(\mc D(u, v)\big)^{-\frac{1}{2\kappa+\delta_0-(n-2)}}\right\}$,
  where we recall that $r_\star = \min\{\frac{1}{\ure}, \frac{R}{2}\}$ and $\mc D(u, v) = \norm[L^2(\Om)]{f}^2+\norm[L^2(\Om)]{g}^2$.
  Then Lemma~\ref{lem:F_inner} and Lemma~\ref{lem:F_outer} imply that there is $c_1>0$ such that
	\begin{align}
	\nn
		-\F(u,v) &=\into uv-\frac12 \into|\nabla v|^2-\frac 12 \into v^2-\into G(u)\\ \nn
		&\le c_1\norm[L^2(\Om)]{f}^{\frac{2n+4}{n+4}}+c_1 r_0^{\delta_0}\norm[L^2(\Om)]{f}^2 +
		c_1 r_0^{\delta_0}\norm[L^2(\Om)]{g}^2 +c_1r_0^{n-2-2\kappa}+c_1  r_0^{-\kappa}\\
	 &\le  c_1 \left( \big(\mc D(u, v)\big)^{\frac{n+2}{n+4}}
		 +r_0^{\delta_0}\big(\mc D(u, v)\big) + 2r_0^{n-2-2\kappa} \right)
    \qquad \text{for all $(u, v) \in \mc S$.}
	\end{align}
	If $ r_\star \le \big(\mc D(u, v)\big)^{-\frac{1}{2\kappa+\delta_0-(n-2)}}$, then $r_0 = r_\star$  and hence
	\begin{align*}
		-\F(u, v) \le c_1 \left( \big(\mc D(u, v)\big)^{\frac{n+2}{n+4}}+ \big(\mc D(u, v)\big)^{\frac{2\kappa-(n-2)}{2\kappa+\delta_0-(n-2)}}+ 2 r_\star^{n-2-2\kappa}\right)
    \qquad \text{for all $(u, v) \in \mc S$},
	\end{align*}
	whereas if $r_\star > \big(\mc D(u, v)\big)^{-\frac{1}{2\kappa+\delta_0-(n-2)}}$, then $r_0= \big(\mc D(u, v)\big)^{-\frac{1}{2\kappa+\delta_0-(n-2)}}$ and
	\begin{align*}
		-\F(u,v) \le c_1 \left(\big(\mc D(u, v)\big)^{\frac{n+2}{n+4}}+ 3 \big(\mc D(u, v)\big)^{\frac{2\kappa-(n-2)}{2\kappa+\delta_0-(n-2)}}\right)
    \qquad \text{for all $(u, v) \in \mc S$}.
	\end{align*}
	Using the fact that $a^x+a^y\le 2 a^{\max\{x,y\}}+1$ for all $a>0$ and all $x,y > 0$,
  we conclude \eqref{eq:F_ge_C_theta:est} for $C \defs \max\{6, 1+ 2 r_\star^{n-2-2\kappa}\} c_1$ and $\gamma \defs \max\{\frac {n+2}{n+4},\frac{2\kappa-(n-2)}{2\kappa+\delta_0-(n-2)}\}\in(0,1)$.
\end{proof}

\section{Initial data leading to blow-up}\label{sec:ftbu_init}
Our next goal, to be achieved in Lemma~\ref{lm:u_v_in_S} below, is to construct families of solutions $(u, v)$ all belonging to the same set $\mc S$ defined in \eqref{eq:def_S}
but being such that the energy of their initial data is not uniformly bounded.

Throughout this section, we let
\begin{align}\label{eq:ass_fd_init}
  \begin{cases}
    n \ge 2,\, R > 0,\, \Omega \defs B_R(0) \subset \R^n, \\
    \phi, \psi \in C^2([0, \infty) \text{ such that \eqref{eq:cond:phi_psi_pos} holds}, \\
    K_{\phi, i}, K_{\psi, i} > 0,\, m \in \R,\, 0 < q_1 \le q_2,\, \maxmq \text{ as in \eqref{eq:cond:m_q1q2_main} and } s_0 > 1 \text{ such that} \\
    \text{\eqref{eq:cond:phi_psi_mq} and \eqref{eq:cond:m_q1q2_main} as well as \eqref{eq:cond:G} (if $n = 2$) or \eqref{eq:cond:G_higher} (if $n \ge 3$) hold}
  \end{cases}
\end{align}
and we first recall a local existence result.
\begin{lemma}\label{lm:local_ex}
  Assume \eqref{eq:ass_fd_init} and let $u_0, v_0 \in \con\infty$ be nonnegative and radially symmetric.
  Then there exist $\tmax \in (0, \infty]$ and a unique pair of nonnegative, radially symmetric functions $(u, v) \in C^\infty(\Ombar \times [0, \tmax))$ solving \eqref{prob:proto} classically
  with the property that if $\tmax < \infty$, then $\limsup_{t \nea \tmax} \|u(\cdot, t)\|_{\leb\infty} = \infty$.
\end{lemma}
\begin{proof}
  The unique existence of nonnegative solutions follows from \cite[Theorem~14.4, Corollary~14.7, Theorem~15.1]{AmannNonhomogeneousLinearQuasilinear1993},
  radial symmetry follows due to uniqueness,
  and the extensibility criterion can be shown as in \cite{CieslakGlobalExistenceSolutions2008}, for instance.
\end{proof}

Since the functions constructed in the following lemma are smooth,
it is not necessary to include a local existence theory for less regular initial data here; Lemma~\ref{lm:local_ex} suffices.
\begin{lemma}\label{lm:ex_init}
  Assume \eqref{eq:ass_fd_init} and let $M_u > 0$.
  Then there exist $\eta_0 > 0$, a family $(u_\eta, v_\eta)_{\eta \in (0, \eta_0)} \in (C^\infty(\Ombar))^2$ of (pairs of) positive, radially symmetric functions,
  $p \in (1, \frac{n}{n-1})$, $\beta \in (\frac np, n)$,
  \begin{align}\label{eq:ex_init:cond_alpha}
    \alpha \in
    \begin{cases}
      {(n, \infty)\vphantom{\left(\frac2q\right)} }, & \maxmq \ge 1, \\
      \left(\frac{n\beta}{(n(m-q_2) + 1)_+}, \frac{2}{\onemaxmq}\right), & \maxmq < 1,
    \end{cases}
  \end{align}
  and $L > 0$ with the following properties:
  For all $\eta \in (0, \eta_0)$,
  \begin{align}\label{eq:ex_init:u_nabla_v_pw_v_w1p}
    \intom u_\eta = M_u, \quad
    \||x|^\alpha u_\eta\|_{\leb\infty} \le L, \quad
    \|v_\eta\|_{\sob1{p}} \le L
    \quad \text{and} \quad
    \||x|^\beta v_\eta\|_{\sob1{\infty}} \le L
  \end{align}
  hold, and the functional $\mc F$ defined in \eqref{eq:def_F} satisfies
  \begin{align}\label{eq:ex_init:F_-infty}
    \mc F(u_\eta, v_\eta) \to -\infty \qquad \text{as $\eta \sea 0$}.
  \end{align}
\end{lemma}
\begin{proof}
  We follow the construction in \cite[Lemma~4.1]{WinklerDoesVolumefillingEffect2009} but also make sure that the quantities appearing in \eqref{eq:ex_init:u_nabla_v_pw_v_w1p} are bounded uniformly in $\eta$.

  If $\maxmq \ge 1$, we take arbitrary $p \in (1, \frac{n}{n-1})$, $\beta \in (\frac np, n)$, $\alpha > n$.
  Else, the key condition \eqref{eq:cond:m_q1q2_main} implies $\frac{n(n-1)}{(n(m-q_2) + 1)_+}< \frac{2}{\onemaxmq}$,
  so that we may fix $p \in (1, \frac{n}{n-1})$, $\beta \in (\frac np, n)$ and $\alpha$ satisfying \eqref{eq:ex_init:cond_alpha}.
  This entails $n(m-q_2) + 1 > 0$, so that \eqref{eq:cond:m_q2} asserts $\alpha > \frac{n \beta}{n-1} > n$.

  For $0 < \eta < \eta_0 \defs \min\{1, \frac R2\}$, we then set
  \begin{align*}
    u_\eta(x) = M_u a_\eta (|x|^2 + \eta^2)^{-\frac{\alpha}{2}}
    \quad \text{for $x \in \Ombar$,  where} \quad
    a_\eta \defs \left(\intom (|y|^2 + \eta^2)^{-\frac{\alpha}{2}} \dy\right)^{-1}.
  \end{align*} 
  The identity $\intom u_\eta = M_u$ as well as positivity and smoothness of $u_\eta$ for all $\eta \in (0, \eta_0)$ follow directly from the definition.
  Moreover,
  \begin{align}\label{eq:ex_init:u_pw}
        u_\eta(x)
    \le M_u a_\eta (|x|^{2})^{-\frac{\alpha}{2}}
    \le M_u a_{\eta_0} |x|^{-\alpha}
    \qquad \text{holds for all $x \in \Omega \setminus \{0\}$ and all $\eta \in (0, \eta_0)$.}
  \end{align}

  If $n = 2$, we let $\gamma \in (0, 1-\theta)$ (and recall that $\theta \in (0, 1)$ is such that \eqref{eq:cond:G} holds)
  and set
  \begin{align*}
    v_\eta(x) = \left( \ln \frac{R}{\eta} \right)^{-\gamma} \ln \frac{2R^2}{|x|^2 + \eta^2}
    \qquad \text{for all $x \in \Ombar$ and all $\eta \in (0, \eta_0)$}.
  \end{align*} 
  For all $\tilde \beta > 1$, we see that there is $c_1 > 0$ such that
  \begin{align*}
        \|r^{\tilde \beta-1} v_\eta\|_{L^\infty((0, R))}
    =   \sup_{r \in (0, R)} r^{\tilde \beta - 1} \left( \ln \frac{R}{\eta} \right)^{-\gamma} \ln \frac{2R^2}{r^2 + \eta^2}
    \le \sup_{r \in (0, R)} r^{\tilde \beta - 1} \left( \ln \frac{R}{\eta_0} \right)^{-\gamma} \ln \frac{2R^2}{r^2}
    \le c_1
  \end{align*}
  and
  \begin{align*}
        \|r^{\tilde \beta} v_{\eta r}\|_{L^\infty((0, R))}
    =   \sup_{r \in (0, R)} \left(r^{\tilde \beta} \left(\ln\frac{R}{\eta}\right)^{-\gamma} \frac{2r}{r^2+\eta^2} \right)
    \le 2 \sup_{r \in (0, R)} r^{\tilde \beta-1} \left(\ln\frac{R}{\eta}\right)^{-\gamma}
    =   2 R^{\tilde \beta-1} \left(\ln\frac{R}{\eta_0}\right)^{-\gamma}
  \end{align*}
  for all $\eta \in (0, \eta_0)$.

  If $n \ge 3$, we let $\delta > n-1 > \frac{n}{2}$ and $\gamma\in ((1-\theta)n, n-2)$ (where $\theta \in (\frac2n, 1)$ is such that \eqref{eq:cond:G_higher} holds)
  and set
  \begin{align*}
    v_\eta(x) = \eta^{\delta-\gamma} (|x|^2+\eta^2)^{-\frac{\delta}{2}}
    \qquad \text{for $x \in \Ombar$ and $\eta \in (0, \eta_0)$}.
  \end{align*} 
  For all $\tilde{\beta} \in (n-1, n)$, we have $\tilde \beta - 1 \in (\gamma, \delta)$ and hence
  \begin{align*}
          \|r^{\tilde \beta-1} v_\eta\|_{L^\infty((0, R))}
    &=    \sup_{r \in (0, R)}  \eta^{\delta-\gamma}  r^{\tilde \beta - 1} (r^2 + \eta^2)^{-\frac{\delta}{2}} \\
    &=    \sup_{r \in (0, R)}  \eta^{\delta-\gamma} \cdot \frac{r^{\tilde \beta - 1}}{(r^2+\eta^2)^{\frac{\tilde \beta - 1}{2}}} \cdot \frac{1}{(r^2+\eta^2)^{\frac{\delta-(\tilde \beta - 1)}{2}}}
     \le  \eta^{\delta-\gamma} \cdot 1 \cdot \eta^{\tilde \beta - 1-\delta}
     \le  \eta_0^{\tilde \beta - 1 - \gamma}
  \end{align*} 
  as well as
  \begin{align*}
      \|r^{\tilde \beta} v_{\eta r}\|_{L^\infty((0, R))}
    &=   \sup_{r \in (0, R)}  \delta \eta^{\delta-\gamma} r^{\tilde \beta+1}  (r^2+\eta^2)^{-\frac{\delta}{2}-1} \\
    &=   \sup_{r \in (0, R)}  \delta \eta^{\delta-\gamma} \cdot \frac{r^{\tilde \beta + 1}}{(r^2+\eta^2)^{\frac{\tilde \beta + 1}{2}}} \cdot \frac{1}{(r^2+\eta^2)^{\frac{\delta- (\tilde \beta - 1)}{2}}}
     \le \delta \eta^{\delta-\gamma} \cdot 1 \cdot \eta^{\tilde \beta - 1-\delta}
     \le \delta \eta_0^{\tilde \beta - 1 - \gamma}
  \end{align*}
  for all $\eta \in (0, \eta_0)$.

  In both cases, we see that $v_\eta$ is positive and smooth for all $\eta \in (0, \eta_0)$ and that $(|x|^\beta v_\eta)_{\eta \in (0, \eta_0)}$ is uniformly bounded in $\sob1\infty$.
  Moreover, choosing an arbitrary $\tilde \beta \in (n-1, \frac np)$, we have $n-1 - p \tilde \beta > -1$ and hence there exists $c_2 > 0$ such that
  \begin{align*}
          \intom v_\eta^p + \intom |\nabla v_\eta|^p 
    &=    \omega_n \int_0^R r^{n-1} \big( v_\eta^p + |v_{\eta r}|^p \big) \dr \\
    &\le  \omega_n \left(R^p \|r^{\tilde \beta-1} v_{\eta}\|_{L^\infty((0, R))}^p + \|r^{\tilde \beta} v_{\eta r}\|_{L^\infty((0, R))}^p \right) \int_0^R r^{n-1 - p \tilde \beta} \dr
    \le   c_2 
  \end{align*}
  for all $\eta \in (0, \eta_0)$.
  Together with \eqref{eq:ex_init:u_pw}, these estimates imply \eqref{eq:ex_init:u_nabla_v_pw_v_w1p} for some $L > 0$ not depending on $\eta$.

  Finally, the proof of \cite[Lemma~4.1]{WinklerDoesVolumefillingEffect2009}, which makes use of the fact that $\alpha > n$ as well as of \eqref{eq:cond:G} and \eqref{eq:cond:G_higher}, shows \eqref{eq:ex_init:F_-infty}.
\end{proof}

We now show that solutions emanating from the initial data constructed in Lemma~\ref{lm:ex_init} indeed belong to the same set $\mc S$ defined in \eqref{eq:def_S}.
To that end, we make crucial use of the pointwise upper estimates results (from \cite{FuestBlowupProfilesQuasilinear2020}; see Appendix~\ref{sec:pw} below).
\begin{lemma}\label{lm:u_v_in_S}
  Assume \eqref{eq:ass_fd_init} and let $M_u > 0$.
  Then we can find $M_v, A, B, \alpha, \kappa$ satisfying \eqref{eq:params_for_S}
  such that with $\eta_0$ and $(u_\eta, v_\eta)_{\eta \in (0, \eta_0)}$ given by Lemma~\ref{lm:ex_init},
  for all $\eta \in (0, \eta_0)$ the solution $(u, v)$ of \eqref{prob:proto} with initial data $u_0 = u_\eta$ and $v_0 = v_\eta$ and maximal existence time $\tmax$ given by Lemma~\ref{lm:local_ex}
  satisfies $(u(\cdot, t), v(\cdot, t)) \in \mc S = \mc S(M_u, M_v, A, B, \alpha, \kappa)$ for all $t \in (0, \tmax)$, where $\mc S$ is as in \eqref{eq:def_S}.
\end{lemma}
\begin{proof}
  We let $p, \beta, \alpha, L, \eta_0$ as well as $(u_\eta, v_\eta)_{\eta \in (0, \eta_0)}$ be as given by Lemma~\ref{lm:ex_init}
  and fix their corresponding solution $(u, v)$ of \eqref{prob:proto} given by Lemma~\ref{lm:local_ex}.
  If $\maxmq \ge 1$, \eqref{eq:upper_u} is a void condition (and we may choose an arbitrary $A > 0$).
  Else, \eqref{eq:cond:phi_psi_mq},
  the lower bound for $\alpha$ in \eqref{eq:ex_init:cond_alpha},
  the uniform estimates in \eqref{eq:ex_init:u_nabla_v_pw_v_w1p}
  and the fact that \eqref{eq:cond:m_q1q2_main} and \eqref{eq:cond:m_q2} entail $m > \frac{n-2}{n}$
  render Theorem~\ref{th:pw_ks} applicable,
  which warrants that there is $A > 0$ (depending only on the quantities in \eqref{eq:ass_fd_init} and on $M_u$, $p$, $\beta$, $\alpha$, $L$) such that
  \begin{align*}
    u(x, t) + 1 \le A |x|^{-\alpha}
    \qquad \text{for all $x \in \Omega$ and all $t \in (0, \tmax)$}.
  \end{align*}

  By Hölder's inequality, $\|\varphi\|_{\leb 1} \le c_1 \|\varphi\|_{\sob1p}$ for all $\varphi \in \sob1p$ and some $c_1 > 0$.
  Since integrating the first equation in \eqref{prob:proto} shows that $\intom u = \intom u_0 = M_u$ in $(0, \tmax)$,
  the comparison principle asserts $\|v(\cdot, t)\|_{\leb 1} \le \max\{\|v_0\|_{\leb1}, M_u\} \le \max\{c_1 L, M_u\} \sfed M_v$.

  Due to \eqref{eq:ex_init:u_nabla_v_pw_v_w1p}, Lemma~\ref{lm:heat_inhom_v_pw} finally asserts that there is $c_2 > 0$ such that $v(x, t) \le c_2 |x|^{-\frac{n-p}{p}}$
  for all $x \in \Omega$ and all $t \in (0, \tmax)$ and hence
  \begin{align*}
    v(x, t) \le B |x|^{-\kappa}
    \qquad \text{for all $x \in \Omega$ and all $t \in (0, \tmax)$},
  \end{align*}
  where $B \defs c_2 R^{n-\frac{n-p}{p}}$ and $\kappa \defs n$.
  We conclude $(u(\cdot, t), v(\cdot, t)) \in \mc S(M_u, M_v, A, B, \alpha, \kappa)$ for all $t \in (0, \tmax)$.
\end{proof}

\section{Conclusion: finite-time blow-up}\label{sec:ftbu_final}
Combining the results from the previous two sections shows that some solutions blow up in finite time whenever the conditions of Theorem~\ref{th:ftbu:q1q2} are met.
\begin{proof}[Proof of Theorem~\ref{th:ftbu:q1q2}]
  We first note that the assumptions of Theorem~\ref{th:ftbu:q1q2} contain both \eqref{eq:ass_fd} and \eqref{eq:params_for_S}.
  Let $M_u>0$. We apply Lemma~\ref{lm:u_v_in_S} to fix $M_v,A,B,\alpha,\kappa > 0$ and find a family of initial data 
  $(u_\eta,v_\eta)_{\eta\in(0,\eta_0)}$ satisfying \eqref{eq:ex_init:F_-infty} and whose corresponding solutions $(u,v)$ given by Lemma~\ref{lm:local_ex} belong to $\mc S = \mc S(M_u,M_v,A,B,\alpha,\kappa)$.
  By Lemma~\ref{lm:F_ge_C_theta}, there are $\gamma\in(0,1)$ and $c_1 > 0$ such that
  \begin{align}\label{eq:f_d_est}
    \F(u,v)\ge -c_1 (\D^{\gamma}(u,v)+1)
    \qquad \text{for all $(u,v)\in \mc S$},
  \end{align}
  where $\F$ and $\D$ are as in \eqref{eq:def_F} and \eqref{eq:def_D}, respectively.
  Moreover, \eqref{eq:ex_init:F_-infty} allows us to fix $\eta>0$ so small such that $\F (u_\eta,v_\eta)< -2c_1$.
  
  We now claim that the maximal existence time $\tmax$ of the solution $(u, v)$ of \eqref{prob:proto} with initial data $(u_\eta, v_\eta)$ is finite.
  Indeed, since a direct computation (see for instance \cite[Lemma~2.1]{WinklerDoesVolumefillingEffect2009}) entails
  \begin{align*}
    \ddt \F(u, v) = - \D(u, v) \le 0
    \qquad \text{in $(0, \tmax)$}
  \end{align*}
  (and hence in particular $\F(u, v) \le -2c_1$ in $(0, \tmax)$),
  and due to \eqref{eq:f_d_est}, we conclude that 
  \begin{align*}
      \ddt (-\F(u, v))
    = \D(u, v)
    \ge \left(\frac{1}{c_1} (-\F(u, v)) - 1 \right)^\frac{1}{\gamma}
    \ge \left(\frac{1}{2c_1} (-\F(u, v)) \right)^\frac{1}{\gamma}
    \qquad \text{in $(0, \tmax)$}.
  \end{align*}
  By comparison, we therefore have
  \begin{align*}
        -\F(u(\cdot, t), v(\cdot, t))
    \ge 2c_1 \left( 1 - \frac{\lambda t}{2c_1} \right)^{-\frac{1}{\lambda}}
    \qquad \text{for all $t \in (0, \tmax)$},
  \end{align*}
  where $\lambda \defs \frac{1-\gamma}{\gamma} > 0$,
  which would be absurd if $\tmax > \frac{2c_1}{\lambda}$.
  Hence $\tmax$ is finite and the extensibility criterion in Lemma~\ref{lm:local_ex} asserts $\limsup_{t \nea \tmax} \|u(\cdot, t)\|_{\leb\infty} = \infty$.
\end{proof}

\begin{proof}[Proof of Theorem~\ref{th:ftbu}]
  That the functions $\phi$ and $\psi$ treated in Theorem~\ref{th:ftbu} are admissible choices for Theorem~\ref{th:ftbu:q1q2} has been shown in Lemma~\ref{lem:phi_psi_admissible}.
\end{proof}

\appendix
\section{Pointwise upper estimates}\label{sec:pw}
A crucial new ingredient in the proof of Theorem~\ref{th:ftbu} compared to \cite{CieslakStinnerFinitetimeBlowupGlobalintime2012,CieslakStinnerFiniteTimeBlowupSupercritical2014,CieslakStinnerNewCriticalExponents2015} consists of pointwise upper estimates for $u$,
which are provided by \cite[Theorem~1.3]{FuestBlowupProfilesQuasilinear2020}.
Unfortunately, in that theorem the estimates depend on $\|v_0\|_{\sob1\infty}$, and the $\sob1\infty$ norm of the functions $v_\eta$ constructed in Lemma~\ref{lm:ex_init} blows up as $\eta \sea 0$.
In this appendix, we therefore prove the following version of \cite[Theorem~1.3]{FuestBlowupProfilesQuasilinear2020},
which relaxes the condition on $v_0$ (and also drops the condition $m \gt \frac{n-2}{n}$ if $\ol T < \infty$) from (1.14) in \cite{FuestBlowupProfilesQuasilinear2020}).
\begin{theorem}\label{th:pw_ks}
  Let
  \begin{align}\label{eq:pw_ks:params}
    n \ge 2, R \gt 0, \Omega \defs B_R(0) \subset \R^n, m, q \in \R, L \gt 0, K_{D, 1}, K_{D, 2}, K_S > 0, d \ge 0, \ol T \in (0, \infty]
  \end{align}
  be such that
  \begin{align*}
    m > \frac{n-2}{n} \quad \text{or} \quad  \ol T < \infty
  \end{align*}
  and
  \begin{align}\label{eq:ps_ks:cond_m_q}
    m-q \in \left(-\frac1n, \frac{n-2}{n} \right].
  \end{align}
  Then for all
  \begin{align}\label{eq:pw_ks:cond_alpha_beta}
    \alpha \gt \frac{n \beta}{(m-q)n + 1}, \quad
    \beta \gt \frac{n}{p}, \quad
    p \in \left[1, \frac{n}{n-1}\right)
    \quad \text{and} \quad
    T \le \ol T,
  \end{align}
  there exists $C \gt 0$ with the following property:
  For any pair of nonnegative, radially symmetric functions $(u_0, v_0) \in \con0 \times \sob1p$
  with
  \begin{align}\label{eq:pw_ks:cond_init_u}
    \intom u_0 \le L,
    \quad \text{and} \quad
    u_0(x) \le L |x|^{-\alpha} \text{ for all } x \in \Omega
  \end{align}
  as well as $|x|^\beta v_0 \in \sob1\infty$,
  \begin{align}\label{eq:pw_ks:cond_init_v}
    \|v_0\|_{\sob1p} \le L \quad \text{and} \quad \||x|^\beta v_0\|_{\sob1\infty} \le L,
  \end{align}
  and any pair of nonnegative, radially symmetric functions
  \begin{align}\label{eq:pw_ks:reg_uv}
    (u, v) \in \left( C^0(\Ombar \times [0, T)) \cap C^{2, 1}(\Ombar \times (0, T))\right)^2
  \end{align}
  solving
  \begin{align}\label{prob:ks_DS}
    \begin{cases}
      u_t = \nabla \cdot (D(u, v) \nabla u - S(u, v) \nabla v)                       & \text{in $\Omega \times (0, T)$}, \\
      v_t = \Delta v - v + u                                                         & \text{in $\Omega \times (0, T)$}, \\
      \partial_\nu u = \partial_\nu v = 0                                            & \text{on $\partial \Omega \times (0, T)$}, \\
      u(\cdot, 0) = u_0, v(\cdot, 0) = v_0                                           & \text{in $\Omega$}
    \end{cases}
  \end{align}
  classically,
  with $D, S \in C^1(Q_u \times Q_v)$,
  $Q_u \defs \ol{u(\Ombar \times [0, T))}$ and $Q_v \defs \ol{v(\Ombar \times [0, T))}$, satisfying
  \begin{align*}
    K_{D, 1} (\rho+d)^{m-1} \le D(\rho, \sigma) \le K_{D, 2} \max\{\rho+d, 1\}^{m-1}
    \quad \text{and} \quad
    |S(\rho, \sigma)| \le K_s \max\{\rho+d, 1\}^q
  \end{align*}
  for all $(\rho, \sigma) \in (Q_u, Q_v)$,
  fulfills
  \begin{align} \label{eq:pw_ks:u_est}
    u(x, t) \le C |x|^{-\alpha}
    \quad \text{and} \quad
    |\nabla v(x, t)| \le C |x|^{-\beta}
    \qquad \text{for all $x \in \Omega$ and all $t \in (0, T)$}.
  \end{align}
\end{theorem}

\subsection{The superfluous condition \tops{$m > \frac{n-2}{n}$}{m > (n-2)/n}}
While revisiting the proof of \cite{FuestBlowupProfilesQuasilinear2020}, we realized that the condition $m > \frac{n-2}{n}$ imposed there is not needed if $\ol T < \infty$.
Although this improvement is inconsequential for the proof of Theorem~\ref{th:ftbu}, we nonetheless briefly illustrate how to overcome this condition.
That is, we show that the following generalization of \cite[Theorem~1.1]{FuestBlowupProfilesQuasilinear2020} holds.
\begin{theorem}\label{th:pw_scalar}
  Suppose that $\Omega \subset \R^n$, $n \ge 2$, is a smooth, bounded domain with $0 \in \Omega$,
  that the parameters
  \begin{align}\label{eq:pw_scalar:params}
    m, q \in \R,
    K_{D,1}, K_{D,2}, K_S, K_f, L, M, \beta \gt 0,
    d \ge 0,
    \ol T \in (0, \infty],
    \theta \gt n,
    \pu \ge 1
  \end{align}
  fulfill
  \begin{align}\label{eq:pw_scalar:T_finite}
    m > \frac{n-2\pu}{n} \quad \text{or} \quad \ol T < \infty
  \end{align}
  as well as
  \begin{align}\label{eq:pw_scalar:cond_m_q}
    m-q \in \left(\frac{\pu}\theta -\frac{\pu}n, \frac{\pu}{\theta} + \frac{\beta \pu-\pu}{n} \right],
  \end{align}
  and that
  \begin{align}\label{eq:pw_scalar:cond_alpha}
    \alpha \gt\frac{\beta}{m-q + \frac{\pu}{n} - \frac{\pu}{\theta}}.
  \end{align}
  Then there exists $C \gt 0$ such that given $T \le \ol T$, $Q_u \subseteq [0, \infty)$,
  \begin{align} \label{eq:pw_scalar:reg_dsf}
    D, S \in C^1(\Ombar \times (0, T) \times Q_u), \quad
    f \in C^1(\Ombar \times (0, T); \R^n)
    \quad \text{and} \quad
    u_0 \in \con0
  \end{align}
  satisfying (with $Q_T \defs \Omega \times (0, T)$)
  \begin{align} 
    \label{eq:pw_scalar:D1}
      &\inf_{(x, t) \in Q_T} D(x, t, \rho) \ge K_{D,1} (\rho+d)^{m-1}, \\
    \label{eq:pw_scalar:D2}
      &\sup_{(x, t) \in Q_T} D(x, t, \rho) \le K_{D,2} \max\{\rho+d, 1\}^{m-1}, \\
    \label{eq:pw_scalar:S}
      &\sup_{(x, t) \in Q_T} |S(x, t, \rho)| \le K_S \max\{\rho+d, 1\}^q
  \end{align}
  for all $\rho \in Q_u$
  and
  \begin{align} \label{eq:pw_scalar:f}
    \sup_{t \in (0, T)} \intom |x|^{\theta \beta} |f(x, t)|^\theta \dx \le K_f
  \end{align}
  as well as
  \begin{align} \label{eq:pw_scalar:u0}
    u_0(x) \le L |x|^{-\alpha} \qquad \text{for all $x \in \Omega$},
  \end{align}
  all nonnegative classical solutions $u \in C^0(\Ombar \times [0, T)) \cap C^{2, 1}(\Ombar \times (0, T))$ of
  \begin{align} \label{prob:scalar_eq}
    \begin{cases}
      u_t \le \nabla \cdot (D(x, t, u) \nabla u + S(x, t, u) f(x, t)), & \text{in $\Omega \times (0, T)$}, \\
      (D(x, t, u) \nabla u + S(x, t, u) f) \cdot \nu \le 0,            & \text{on $\partial \Omega \times (0, T)$}, \\
      u(\cdot, 0) \le u_0,                                             & \text{in $\Omega$},
    \end{cases}
  \end{align}
  with $\ol{u(\Ombar \times [0, T))} \subseteq Q_u$ and
  \begin{align} \label{eq:pw_scalar:p}
    \sup_{t \in (0, T)} \intom u^{\pu}(\cdot, t) \le M
  \end{align}
  fulfill the first estimate in \eqref{eq:pw_ks:u_est}.
\end{theorem}

Throughout this subsection, we henceforth fix a smooth, bounded domain $\Omega \subset \R^n$, $n \ge 2$, with $0 \in \Omega$
as well as parameters in \eqref{eq:pw_scalar:params} and $\alpha$ satisfying \eqref{eq:pw_scalar:cond_m_q} and \eqref{eq:pw_scalar:cond_alpha}
(where we assume $\beta \gt (m-q) \alpha$ without loss of generality).
For the sake of clarity, we also fix a set $Q_u \subseteq [0, \infty)$, functions in \eqref{eq:pw_scalar:reg_dsf} complying with \eqref{eq:pw_scalar:D1}--\eqref{eq:pw_scalar:u0}
and a solution $u \in C^0(\Ombar \times [0, T)) \cap C^{2, 1}(\Ombar \times (0, T))$ of \eqref{prob:scalar_eq}
satisfying \eqref{eq:pw_scalar:p} and $\ol{u(\Ombar \times [0, T))} \subseteq Q_u$,
but stress that all constants below only depend on the parameters in \eqref{eq:pw_scalar:params} and on $\alpha$.

We first note that instances where the parameter $d$ in \eqref{eq:pw_scalar:params} is positive may be reduced to the case $d = 0$ considered in \cite{FuestBlowupProfilesQuasilinear2020}.
\begin{lemma}\label{lm:d_gt_0_red}
  Suppose that Theorem~\ref{th:pw_scalar} holds under the restriction that the parameter $d$ in \eqref{eq:pw_scalar:params} vanishes.
  Then that theorem is true in general.
\end{lemma}
\begin{proof}
  The function $\wt u \defs u + d$ fulfills
  \begin{align*}
    \begin{cases}
      \wt u_t \le \nabla \cdot ( D(x, t, \wt u - d) \nabla \wt u + S(x, t, \wt u - d) f ) & \text{in $\Omega \times (0, T)$}, \\
      D(x, t, \wt u - d) \nabla \wt u + S(x, t, \wt u - d) f \le 0                           & \text{in $\partial \Omega \times (0, T)$}, \\
      \wt u(\cdot, 0) \le u_0 + d                                                                  & \text{in $\Omega$}.
    \end{cases}
  \end{align*}
  We set
  \begin{align*}
    \wt D(x, t, \rho) = D(x, t, \rho - d)
    \quad \text{and} \quad
    \wt S(x, t, \rho) = S(x, t, \rho - d)
  \end{align*}
  for $x \in \Ombar$, $t \in (0, T)$ and $\rho \in Q_{\tilde u} \defs Q_{u} + d \subseteq [d, \infty)$.
  Noting that \eqref{eq:pw_scalar:D1}--\eqref{eq:pw_scalar:S} with $d = 0$ hold for these $x, t, \rho$ when $(D, S)$ is replaced by $(\wt D, \wt S)$,
  we may apply Theorem~\ref{th:pw_scalar} with $d = 0$ to $\tilde u$,
  which yields $c_1 > 0$ with $u(x, t) \le \tilde u(x, t) \le c_1 |x|^{-\alpha}$ for all $x \in \Omega$ and all $t \in (0, T)$.
\end{proof}
 
Since \cite{FuestBlowupProfilesQuasilinear2020} proves Theorem~\ref{th:pw_scalar} for $m > \frac{n-2\pu}{n}$ and recalling \eqref{eq:pw_scalar:T_finite},
we may without loss of generality thus not only assume $d = 0$ but also $\ol T < \infty$ and $m < 1$.
Moreover, as in \cite{FuestBlowupProfilesQuasilinear2020},
we set $w(x, t) \defs |x|^\alpha u(x, t)$, aiming to obtain an $\leb\infty$ bound for $w$
which implies the desired estimate \eqref{eq:pw_ks:u_est}.
We start by stating the eventual outcome of a testing procedure performed in \cite{FuestBlowupProfilesQuasilinear2020}.
\begin{lemma}\label{lm:w_bdd_lp}
  There are $\tilde p \ge 1$, $\nu \gt 1$ and $C \gt 0$ with the property that
  \begin{align}\label{eq:w_bdd_lp:bdd}
    \sup_{t \in (0, T)} \intom w^p(\cdot, t) \le C
  \end{align}
  and
  \begin{align}\label{eq:w_bdd_lp:ddt}
        \ddt \intom w^p
    \le C p^\nu + C p^\nu \left( \intom w^{\frac{p}{2}} \right)^2
    \qquad \text{in $(0, T)$}
  \end{align}
  for all $p \ge \tilde p$.
\end{lemma}
\begin{proof}
  Following \cite[Section~2]{FuestBlowupProfilesQuasilinear2020}, we see that for all $s \in (0, \frac{2n}{n-2})$,
  there are $\tilde p \ge 3-m$, $\nu > 1$ and $c_1 > 0$ such that for all $p \ge \tilde p$,
  \begin{align}\label{eq:w_bdd_lp:ddt1}
        \ddt \intom w^p
    \le c_1 p^\nu + c_1 p^\nu \left( \intom w^{(p+m-1)s - 1} \right)^\frac1s
    \qquad \text{in $(0, T)$}.
  \end{align}
  Indeed, there are only two differences compared to \cite{FuestBlowupProfilesQuasilinear2020}.
  First, in \cite{FuestBlowupProfilesQuasilinear2020} only $Q_u = [0, \infty)$ is considered.
  However, it is evidently sufficient to require \eqref{eq:pw_scalar:D1}--\eqref{eq:pw_scalar:S} to hold for $\rho = u(x, t)$ only.
  Second, we do not require $m > \frac{n-2\pu}{n}$,
  meaning that we cannot make use of \cite[Lemma~2.7]{FuestBlowupProfilesQuasilinear2020}, i.e.,
  of estimates of the form $\intom w^p \le \eps \intom |\nabla \varphi|^2 + C (\intom \varphi^{\tilde s})^\frac1{\tilde s}$,
  where $\varphi = |x|^{-\frac{(m-1)\alpha}{2}} w^\frac{p+m-1}{2}$.
  This has only been used in \cite[Lemma~2.8]{FuestBlowupProfilesQuasilinear2020} to obtain \eqref{eq:w_bdd_lp:ddt1} with an additional summand $\intom w^p$ on the left-hand side.
  Not containing such a dissipation term, \eqref{eq:w_bdd_lp:ddt1} is hence also valid for $m \le \frac{n-2\pu}{n}$.
  
  A first application of \eqref{eq:w_bdd_lp:ddt1} with $s = \frac{2}{p+m-1} > 0$ yields \eqref{eq:w_bdd_lp:bdd},
  as \eqref{eq:pw_scalar:p} entails $\sup_{t \in (0, T)} \|w(\cdot, t)\|_{\leb1} \le \max\{|x|^\alpha \mid x \in \Ombar\} M^\frac{1}{\pu} |\Omega|^{1+\frac{\pu-1}{\pu}}$ and since $T$ is finite.
  Moreover, by taking $s = \frac12$ in \eqref{eq:w_bdd_lp:ddt1}, we conclude that
  \begin{align*}
          \ddt \intom w^p - c_1 p^\nu
    &\le  c_1 p^\nu \left( \intom w^{\frac{p+m-1}{2}-1} \right)^2
    \le   c_1 p^\nu \left( \intom (w+1)^\frac{p}{2} \right)^2
    \le   2^\frac{p-2}{2} c_1 p^\nu \left( \intom w^\frac{p}{2} + |\Omega| \right)^2 \\
    &\le  2^\frac{p}{2} c_1 p^\nu \left( \intom w^\frac{p}{2} \right)^2
          + 2^\frac{p}{2} |\Omega| c_1 p^\nu
    \qquad \text{in $(0, T)$ for all $p \ge \tilde p$},
  \end{align*}
  where we have made use of the assumption $m < 1$.
\end{proof}

Unlike \cite[(2.19)]{FuestBlowupProfilesQuasilinear2020}, \eqref{eq:w_bdd_lp:ddt} does not contain any dissipative term,
so that we need to adapt the Moser-type iteration carried out in \cite[Lemma~2.10]{FuestBlowupProfilesQuasilinear2020} to the present situation.
\begin{lemma}\label{lm:w_bdd_l_infty}
  There is $C \gt 0$ such that
  \begin{align} \label{eq:w_bdd_l_infty:statement}
    \|w\|_{L^\infty(\Omega \times (0, T))} \lt C.
  \end{align}
\end{lemma}
\begin{proof}
  Letting $\tilde p \ge 1$ and $\nu \gt 1$ be as given by Lemma~\ref{lm:w_bdd_lp} and setting $w_0 \defs w(\cdot, 0)$,
  we infer from \eqref{eq:w_bdd_lp:ddt} upon integrating that there is $c_1 > 0$ with
  \begin{align*}
          \intom w^p(\cdot, t)
    &\le  \intom w_0^p
          + T c_1 p^\nu
          + c_1 p^\nu \sup_{s \in (0, t)} \int_0^t \left( \intom w^\frac{p}{2}(\cdot, s) \right)^2 \\
    &\le  3 \max\left\{
            \intom w_0^p,\,
            \ol T c_1 p^\nu,\,
            \ol T c_1 p^\nu \sup_{s \in (0, t)} \left( \intom w^\frac{p}{2}(\cdot, s) \right)^2
          \right\}
  \end{align*}
  for all $t \in (0, T)$ and all $p \ge \tilde p$.
  Thus, setting $p_j \defs 2^j \tilde p$, $p_\infty \defs \infty$
  and $A_j \defs \sup_{t \in (0, T)} \|w(\cdot, t)\|_{\leb{p_j}}$ for $j \in \N_0 \cup \{\infty\}$
  as well as $c_2 \defs 3 \ol T c_1$,
  we have
  \begin{align*}
        A_j
    \le \max\left\{
          3^\frac1{p_j} \|w_0\|_{\leb{p_j}},\,
          (c_2 p_j^\nu)^\frac1{p_j},\,
          (c_2 p_j^\nu)^\frac1{p_j} A_{j-1}
        \right\}
    \qquad \text{for all $j \in \N$}.
  \end{align*}
  We note that by \eqref{eq:w_bdd_lp:bdd}, there is $c_3 > 0$ with $A_0 \le c_3$.
  Moreover, as $3^\frac1{p_j} \|w_0\|_{\leb{p_j}} \le 3 \|w_0\|_{\leb \infty} |\Omega|^\frac1{p_j} \le 3L \max\{|\Omega|, 1\}$ for $j \in \N$
  and since $\lim_{p \to \infty} (c_2 p^\nu)^\frac1p = 1$, there is $c_4 > 0$ with 
  $A_j \le \max\left\{c_4, (c_2 p_j^\nu)^\frac1{p_j} A_{j-1} \right\}$ for all $j \in \N$.
   
  Without loss of generality,
  we may assume that there is $j_0 \in \N$ with $A_j \gt c_4$ for all $j \ge j_0$,
  as else there would be a sequence $(j_k)_{k \in \N}$
  with $j_k \ra \infty$ as $k \ra \infty$ and $A_{j_k} \le c_4$ for all $k \in \N$,
  already implying $A_\infty = \limsup_{k \ra \infty} A_{j_k} \le c_4$ and hence \eqref{eq:w_bdd_l_infty:statement}.

  Therefore, $A_j \le (c_2 p_j^\nu)^\frac1{p_j} A_{j-1}$ for all $j \gt j_0$ and, if we take $j_0$ as small as possible, $A_{j_0} \le \max\{c_3, c_4\}$.
  Due to $p_j = 2^j \tilde p \le (2 \tilde p)^j$ for $j \in \N_0$, an induction argument yields
  \begin{align*}
          A_j
    &\le  \left(\prod_{k=j_0+1}^{j} (c_2 p_k^\nu)^\frac1{p_k} \right) A_{j_0} \\
    &\le  c_2^{\tilde p^{-1} \sum_{k=j_0+1}^j 2^{-k}} \cdot (2 \tilde p)^{\tilde p^{-1} \nu \sum_{k=j_0+1}^j k2^{-k}} \cdot A_{j_0} \\
    &\le  c_2^{\tilde p^{-1} \sum_{k=0}^\infty 2^{-k}} \cdot (2 \tilde p)^{\tilde p^{-1} \nu \sum_{k=0}^\infty k2^{-k}} \cdot \max\{c_3, c_4\}
    \sfed c_5
  \end{align*}
  for all $j \gt j_0$ and hence $A_\infty \le c_5$.
\end{proof}

\begin{proof}[Proof of Theorem~\ref{th:pw_scalar}]
  The definition of $w$ asserts $u(x, t) = |x|^{-\alpha} w(x, t) \le |x|^{-\alpha} \|w\|_{L^\infty(\Omega \times (0, t))}$
  for all $x \in \Omega$ and all $t \in (0, T)$, so that the theorem follows from Lemma~\ref{lm:w_bdd_l_infty}.
\end{proof}

\subsection{Relaxed conditions for \tops{$v_0$}{v0}}
Similarly as in \cite{FuestBlowupProfilesQuasilinear2020}, Theorem~\ref{th:pw_ks} will follow from Theorem~\ref{th:pw_scalar}.
That is, we essentially need to show that solutions of the heat equation with a force term controlled in $L^\infty((0, T); \leb1)$ satisfy \eqref{eq:pw_scalar:f} for appropriate parameters
and, importantly, that these estimates do not depend on $\|v_0\|_{\sob1\infty}$ but on $v_0$ only by the quantities in \eqref{eq:pw_ks:cond_init_v}.
To that end, we adapt the reasoning of \cite{WinklerBlowupProfilesLife2020}.

Throughout this subsection, we fix $n \ge 2$, $R > 0$, $\Omega = B_R(0)$,
\begin{equation}\label{eq:heat_inhom:cond_p_beta}
  p \in \left[1, \frac{n}{n-1}\right)
  \quad \text{and} \quad
  \beta > \frac{n}{p}.
\end{equation}
We consider classical solutions
\begin{align*}
  v \in C^0(\Ombar \times [0, T)) \cap C^{2, 1}(\Ombar \times (0, T))
\end{align*}
of
\begin{align}\label{prob:heat_inhom}
  \begin{cases}
    v_t = \Delta v - v + g(x, t) & \text{in $\Omega \times (0, T)$}, \\
    \partial_\nu v = 0           & \text{on $\partial \Omega \times (0, T)$}, \\
    v(\cdot, 0) = v_0            & \text{in $\Omega$},
  \end{cases}
\end{align}
where
\begin{alignat}{2}
   & g \in C^0(\Ombar \times [0, T))     && \quad \text{is radially symmetric and nonnegative with} \quad \sup_{t \in (0, T)} \|g(\cdot, t)\|_{\leb1} \le M, \label{eq:heat_inhom:cond_g}\\
   & v_0 \in \sob1p                      && \quad \text{is radially symmetric, nonnegative and such that} \quad \|v_0\|_{\sob1p} \le M \quad \text{and} \label{eq:heat_inhom:cond_v0_w1p}\\
   & |x|^\beta v_0 \in \sob1\infty       && \quad \text{with} \quad \||x|^\beta v_0\|_{\sob1\infty} \le M. \label{eq:heat_inhom:cond_xv_0_w1infty}
\end{alignat}

\begin{lemma}\label{lm:heat_inhom_w1p}
  There exists $C > 0$ with the following property:
  Let $M > 0$, $T \in (0, \infty]$ and $g, v_0$ as in \eqref{eq:heat_inhom:cond_g} and \eqref{eq:heat_inhom:cond_v0_w1p}.
  Then the classical solution $v$ of \eqref{prob:heat_inhom} fulfills
  \begin{align*}
    \|v(\cdot, t)\|_{\sob1p} \le C M
    \qquad \text{for all $t \in (0, T)$}.
  \end{align*}
\end{lemma}
\begin{proof}
  Due to the restriction $p \in [1, \frac{n}{n-1})$, this can rapidly be seen by means of semigroup arguments,
  for instance based on the $L^p$-$L^q$ estimates collected in \cite[Lemma~1.3]{WinklerAggregationVsGlobal2010}.
\end{proof}

\begin{lemma}\label{lm:heat_inhom_v_pw}
  There exists $C > 0$ with the following property:
  Let $M > 0$, $T \in (0, \infty]$ and $g, v_0$ as in \eqref{eq:heat_inhom:cond_g} and \eqref{eq:heat_inhom:cond_v0_w1p}.
  Then the classical solution $v$ of \eqref{prob:heat_inhom} fulfills
  \begin{align*}
    v(r, t) \le CM r^{-\frac{n-p}{p}}
    \qquad \text{for all $r \in (0, R)$ and all $t \in (0, T)$}.
  \end{align*}
\end{lemma}
\begin{proof}
  This follows by applying minor modifications to the proof of \cite[Lemma~3.2]{WinklerFinitetimeBlowupHigherdimensional2013}:
  Replace $u$ by $g$, apply Lemma~\ref{lm:heat_inhom_w1p} instead of \cite[Lemma~3.1]{WinklerFinitetimeBlowupHigherdimensional2013} in \cite[(3.5)]{WinklerFinitetimeBlowupHigherdimensional2013}
  and note that $\intom v \le \max\{\intom v_0, \sup_{t \in (0, T)} \|g(\cdot, t)\|_{\leb1}\} \le c_1 M$ in $(0, T)$ for some $c_1 > 0$ by the maximum principle.
\end{proof}

\begin{lemma}\label{lm:x_beta_nabla_v_est}
  Let $M > 0$, $\theta \in (1, \infty]$, $T \in (0, \infty]$ and, if $\theta = \infty$, $\tilde \alpha > \beta$.
  Then there exists $C > 0$ with the following property:
  Suppose that $g, v_0$ are as in \eqref{eq:heat_inhom:cond_g} and \eqref{eq:heat_inhom:cond_v0_w1p},
  and that, if $\theta = \infty$, 
  \begin{align*}
    g(x, t) \le M |x|^{-\tilde \alpha} \quad \text{for all $x \in \Omega$ and $t \in (0, T)$}.
  \end{align*}
  Then the classical solution $v$ of \eqref{prob:heat_inhom} fulfills
  \begin{align*}
    \sup_{t \in (0, T)} \||x|^{\beta} \nabla v(\cdot, t)\|_{\leb\theta} \le C.
  \end{align*}
\end{lemma}
\begin{proof}
  Apart from the following two modifications, this can be proven by a near verbatim copy of the proof of \cite[Lemma~3.4]{WinklerBlowupProfilesLife2020}.
  First, we set $\kappa = \frac{n-p}{p}$ and note that by Lemma~\ref{lm:heat_inhom_v_pw}, there is $c_1 > 0$ such that $v(r, t) \le c_1 r^{-\kappa}$ for all $(r, t) \in (0, R) \times (0, T)$.
  For the proof of \cite[Lemma~3.4]{WinklerBlowupProfilesLife2020}, $\kappa$ needs to additionally satisfy the second condition in \cite[(3.14)]{WinklerBlowupProfilesLife2020},
  namely $n - 2 < \kappa < \beta - 1$.
  This holds by \eqref{eq:heat_inhom:cond_p_beta}.

  Second, the only place where conditions for the initial data $v_0$ explicitly enter the proof of \cite[Lemma~3.4]{WinklerBlowupProfilesLife2020} is in \cite[(3.26)]{WinklerBlowupProfilesLife2020}.
  There the term $\|r^\beta v_0\|_{W^{1, \infty}((0, R))}$ appears which \cite{WinklerBlowupProfilesLife2020} controls by a bound for $v_0$ in $\sob1\infty$
  but of course the condition \eqref{eq:heat_inhom:cond_xv_0_w1infty} is sufficient.
\end{proof}

\subsection{Proof of Theorem~\ref{th:pw_ks}}
\begin{proof}[Proof of Theorem~\ref{th:pw_ks}]
  Applying Lemma~\ref{lm:x_beta_nabla_v_est} to $g \defs u$ and Theorem~\ref{th:pw_scalar} to $f \defs -\nabla v$ yields the desired estimates.
\end{proof}

\section*{Acknowledgments}
We thank the reviewer for their constructive comments that significantly enhanced the manuscript.
The first author acknowledges support of the Natural Science Foundation of Shanghai within the project 23ZR1400100.

\section*{Conflict of interest statement}
On behalf of all authors, the corresponding author states that there is no conflict of interest.\\

\section*{Data availability statement}
No new data were created or analysed during this study. Data sharing is not applicable to this article.


\footnotesize
\begin{thebibliography}{10}
\setlength{\itemsep}{0.2pt}

\bibitem{AmannNonhomogeneousLinearQuasilinear1993}
\textsc{Amann, H.}:
\newblock {\em Nonhomogeneous linear and quasilinear elliptic and
  pa\-ra\-bo\-lic boundary value problems}.
\newblock In \textsc{Schmeisser, H.} and \textsc{Triebel, H.}, editors, {\em
  Function {{Spaces}}, {{Differential Operators}} and {{Nonlinear Analysis}}},
  \href{https://doi.org/10.1007/978-3-663-11336-2_1}{pages 9--126}.
  Vieweg+Teubner Verlag, Wiesbaden, 1993.
\newblock

\bibitem{BellomoEtAlMathematicalTheoryKeller2015}
\textsc{Bellomo, N.}, \textsc{Bellouquid, A.}, \textsc{Tao, Y.}, and
  \textsc{Winkler, M.}:
\newblock {\em Toward a mathematical theory of {{Keller}}--{{Segel}} models of
  pattern formation in biological tissues}.
\newblock Math. Models Methods Appl. Sci.,
  \href{https://doi.org/10.1142/S021820251550044X}{25(09):1663--1763}, 2015.
\newblock

\bibitem{BlackEtAlRelaxedParameterConditions2021}
\textsc{Black, T.}, \textsc{Fuest, M.}, and \textsc{Lankeit, J.}:
\newblock {\em Relaxed parameter conditions for chemotactic collapse in
  logistic-type parabolic--elliptic {{Keller--Segel}} systems}.
\newblock Z. Angew. Math. Phys.,
  \href{https://doi.org/10.1007/s00033-021-01524-8}{72(3):Art.~96}, 2021.
\newblock

\bibitem{CieslakGlobalExistenceSolutions2008}
\textsc{Cie{\'s}lak, T.}:
\newblock {\em Global existence of solutions to a chemotaxis system with volume
  filling effect}.
\newblock Colloq. Math.,
  \href{https://doi.org/10.4064/cm111-1-11}{111(1):117--134}, 2008.
\newblock

\bibitem{CieslakLaurencotFiniteTimeBlowup2010}
\textsc{Cie{\'s}lak, T.} and \textsc{Lauren{\c c}ot, {\relax Ph}.}:
\newblock {\em Finite time blow-up for a one-dimensional quasilinear
  parabolic--parabolic chemotaxis system}.
\newblock Ann. Inst. Henri Poincar{\'e} C Anal. Non Lin{\'e}aire,
  \href{https://doi.org/10.1016/j.anihpc.2009.11.016}{27(1):437--446}, 2010.
\newblock

\bibitem{CieslakStinnerFinitetimeBlowupGlobalintime2012}
\textsc{Cie{\'s}lak, T.} and \textsc{Stinner, {\relax Ch}.}:
\newblock {\em Finite-time blowup and global-in-time unbounded solutions to a
  parabolic{\textendash}parabolic quasilinear {{Keller}}{\textendash}{{Segel}}
  system in higher dimensions}.
\newblock J. Differ. Equ.,
  \href{https://doi.org/10.1016/j.jde.2012.01.045}{252(10):5832--5851}, 2012.
\newblock

\bibitem{CieslakStinnerFiniteTimeBlowupSupercritical2014}
\textsc{Cie{\'s}lak, T.} and \textsc{Stinner, {\relax Ch}.}:
\newblock {\em Finite-time blowup in a supercritical quasilinear
  parabolic-parabolic {{Keller}}{\textendash}{{Segel}} system in dimension 2}.
\newblock Acta Appl. Math.,
  \href{https://doi.org/10.1007/s10440-013-9832-5}{129(1):135--146}, 2014.
\newblock

\bibitem{CieslakStinnerNewCriticalExponents2015}
\textsc{Cie{\'s}lak, T.} and \textsc{Stinner, {\relax Ch}.}:
\newblock {\em New critical exponents in a fully parabolic quasilinear
  {{Keller}}{\textendash}{{Segel}} system and applications to volume filling
  models}.
\newblock J. Differ. Equ.,
  \href{https://doi.org/10.1016/j.jde.2014.12.004}{258(6):2080--2113}, 2015.
\newblock

\bibitem{ding2019sigma}
\textsc{Ding, M.} and \textsc{Zhao, X.}:
\newblock {\em {$L^\sigma$}-measure criteria for boundedness in a quasilinear
  parabolic-parabolic {K}eller-{S}egel system with supercritical sensitivity}.
\newblock Discrete Contin. Dyn. Syst. Ser. B,
  \href{https://doi.org/10.3934/dcdsb.2019059}{24(10):5297--5315}, 2019.
\newblock

\bibitem{DuLiuBlowupSolutionsChemotaxis2023}
\textsc{Du, W.} and \textsc{Liu, S.}:
\newblock {\em Blow-up solutions of a chemotaxis model with nonlocal effects}.
\newblock Nonlinear Anal. Real World Appl.,
  \href{https://doi.org/10.1016/j.nonrwa.2023.103890}{73:Art. 103890, 14 pp.},
  2023.
\newblock

\bibitem{FuestBlowupProfilesQuasilinear2020}
\textsc{Fuest, M.}:
\newblock {\em Blow-up profiles in quasilinear fully parabolic
  {{Keller}}{\textendash}{{Segel}} systems}.
\newblock Nonlinearity,
  \href{https://doi.org/10.1088/1361-6544/ab7294}{33(5):2306--2334}, 2020.
\newblock

\bibitem{FuestFinitetimeBlowupTwodimensional2020}
\textsc{Fuest, M.}:
\newblock {\em Finite-time blow-up in a two-dimensional {{Keller}}--{{Segel}}
  system with an environmental dependent logistic source}.
\newblock Nonlinear Anal. Real World Appl.,
  \href{https://doi.org/10.1016/j.nonrwa.2019.103022}{52:Art.~103022, 14 pp.},
  2020.
\newblock

\bibitem{FuestApproachingOptimalityBlowup2021}
\textsc{Fuest, M.}:
\newblock {\em Approaching optimality in blow-up results for
  {{Keller}}--{{Segel}} systems with logistic-type dampening}.
\newblock Nonlinear Differ. Equ. Appl. NoDEA,
  \href{https://doi.org/10.1007/s00030-021-00677-9}{28(2):Art.~16, 17 pp.},
  2021.
\newblock

\bibitem{fuest2022optimality}
\textsc{Fuest, M.}:
\newblock {\em On the optimality of upper estimates near blow-up in quasilinear
  {{Keller}}--{{Segel}} systems}.
\newblock Appl. Anal.,
  \href{https://doi.org/10.1080/00036811.2020.1854234}{101(9):3515--3534},
  2022.
\newblock

\bibitem{FuestLankeitCornersCollapseSimple2023}
\textsc{Fuest, M.} and \textsc{Lankeit, J.}:
\newblock {\em Corners and collapse: {{Some}} simple observations concerning critical
  masses and boundary blow-up in the fully parabolic {{Keller--Segel}}} system.
\newblock Appl. Math. Lett.,
  \href{https://doi.org/10.1016/j.aml.2023.108788}{146:Art.~108788, 9 pp.},
  2023.
\newblock

\bibitem{HerreroVelazquezBlowupMechanismChemotaxis1997}
\textsc{Herrero, M.~A.} and \textsc{Vel{\'a}zquez, J. J.~L.}:
\newblock {\em A blow-up mechanism for a chemotaxis model}.
\newblock Ann. Della Scuola Norm. Super. Pisa Cl. Sci. Ser. IV, 24(4):633--683
  (1998), 1997.

\bibitem{HillenPainterUserGuidePDE2009}
\textsc{Hillen, T.} and \textsc{Painter, K.~J.}:
\newblock {\em A user's guide to {{PDE}} models for chemotaxis}.
\newblock J. Math. Biol.,
  \href{https://doi.org/10.1007/s00285-008-0201-3}{58(1-2):183--217}, 2009.
\newblock

\bibitem{HorstmannWangBlowupChemotaxisModel2001}
\textsc{Horstmann, D.} and \textsc{Wang, G.}:
\newblock {\em Blow-up in a chemotaxis model without symmetry assumptions}.
\newblock Eur. J. Appl. Math.,
  \href{https://doi.org/10.1017/S0956792501004363}{12(02):159--177}, 2001.
\newblock

\bibitem{HorstmannWinklerBoundednessVsBlowup2005}
\textsc{Horstmann, D.} and \textsc{Winkler, M.}:
\newblock {\em Boundedness vs. blow-up in a chemotaxis system}.
\newblock J. Differ. Equ.,
  \href{https://doi.org/10.1016/j.jde.2004.10.022}{215(1):52--107}, 2005.
\newblock

\bibitem{IshidaEtAlBoundednessQuasilinearKeller2014}
\textsc{Ishida, S.}, \textsc{Seki, K.}, and \textsc{Yokota, T.}:
\newblock {\em Boundedness in quasilinear {{Keller}}--{{Segel}} systems of
  parabolic--parabolic type on non-convex bounded domains}.
\newblock J. Differ. Equ.,
  \href{https://doi.org/10.1016/j.jde.2014.01.028}{256(8):2993--3010}, 2014.
\newblock

\bibitem{JagerLuckhausExplosionsSolutionsSystem1992}
\textsc{J{\"a}ger, W.} and \textsc{Luckhaus, S.}:
\newblock {\em On explosions of solutions to a system of partial differential
  equations modelling chemotaxis}.
\newblock Trans. Am. Math. Soc.,
  \href{https://doi.org/10.2307/2153966}{329(2):819--824}, 1992.
\newblock

\bibitem{KellerSegelInitiationSlimeMold1970}
\textsc{Keller, E.~F.} and \textsc{Segel, L.~A.}:
\newblock {\em Initiation of slime mold aggregation viewed as an instability}.
\newblock J. Theor. Biol.,
  \href{https://doi.org/10.1016/0022-5193(70)90092-5}{26(3):399--415}, 1970.
\newblock

\bibitem{LankeitInfiniteTimeBlowup2020}
\textsc{Lankeit, J.}:
\newblock {\em Infinite time blow-up of many solutions to a general quasilinear
  parabolic-elliptic {{Keller-Segel}} system}.
\newblock Discrete Contin. Dyn. Syst. - S,
  \href{https://doi.org/10.3934/dcdss.2020013}{13(2):233--255}, 2020.
\newblock

\bibitem{LankeitWinklerFacingLowRegularity2019}
\textsc{Lankeit, J.} and \textsc{Winkler, M.}:
\newblock {\em Facing low regularity in chemotaxis systems}.
\newblock Jahresber. Dtsch. Math.-Ver.,
  \href{https://doi.org/10.1365/s13291-019-00210-z}{122:35--64}, 2019.
\newblock

\bibitem{LaurencotMizoguchiFiniteTimeBlowup2017}
\textsc{Lauren{\c c}ot, {\relax Ph}.} and \textsc{Mizoguchi, N.}:
\newblock {\em Finite time blowup for the parabolic--parabolic
  {{Keller}}--{{Segel}} system with critical diffusion}.
\newblock Ann. Inst. Henri Poincar{\'e} C Anal. Non Lin{\'e}aire,
  \href{https://doi.org/10.1016/j.anihpc.2015.11.002}{34(1):197--220}, 2017.
\newblock

\bibitem{MarrasEtAlBehaviorTimeSolutions2023}
\textsc{Marras, M.}, \textsc{{Vernier-Piro}, S.}, and \textsc{Yokota, T.}:
\newblock {\em Behavior in time of solutions of a {{Keller-Segel}} system with
  flux limitation and source term}.
\newblock NoDEA Nonlinear Differential Equations Appl.,
  \href{https://doi.org/10.1007/s00030-023-00874-8}{30(5):Art. 65, 27 pp.},
  2023.
\newblock

\bibitem{MizoguchiWinklerBlowupTwodimensionalParabolic}
\textsc{Mizoguchi, N.} and \textsc{Winkler, M.}:
\newblock {\em Blow-up in the two-dimensional parabolic
  {{Keller}}{\textendash}{{Segel}} system}.
\newblock Preprint.

\bibitem{MizukamiEtAlCanChemotacticEffects2022}
\textsc{Mizukami, M.}, \textsc{Tanaka, Y.}, and \textsc{Yokota, T.}:
\newblock {\em Can chemotactic effects lead to blow-up or not in two-species
  chemotaxis-competition models?}
\newblock Z. Angew. Math. Phys.,
  \href{https://doi.org/10.1007/s00033-022-01878-7}{73(6):Art. 239, 25 pp.},
  2022.
\newblock

\bibitem{NagaiBlowupRadiallySymmetric1995}
\textsc{Nagai, T.}:
\newblock {\em Blow-up of radially symmetric solutions to a chemotaxis system}.
\newblock Adv. Math. Sci. Appl., 5(2):581--601, 1995.

\bibitem{NagaiBlowupNonradialSolutions2001}
\textsc{Nagai, T.}:
\newblock {\em Blowup of nonradial solutions to parabolic-elliptic systems
  modeling chemotaxis in two-dimensional domains}.
\newblock J. Inequalities Appl.,
  \href{https://doi.org/10.1155/S1025583401000042}{6(1):37--55}, 2001.
\newblock

\bibitem{nagai2000chemotactic}
\textsc{Nagai, T.}, \textsc{Senba, T.}, and \textsc{Suzuki, T.}:
\newblock {\em Chemotactic collapse in a parabolic system of mathematical
  biology}.
\newblock Hiroshima Math. J.,
  \href{https://doi.org/10.32917/hmj/1206124609}{30(3):463--497}, 2000.
\newblock

\bibitem{NagaiEtAlApplicationTrudingerMoserInequality1997}
\textsc{Nagai, T.}, \textsc{Senba, T.}, and \textsc{Yoshida, K.}:
\newblock {\em Application of the {{Trudinger}}--{{Moser}} inequality to a
  parabolic system of chemotaxis}.
\newblock Funkc. Ekvacioj, 40:411--433, 1997.

\bibitem{PainterHillenVolumefillingQuorumsensingModels2002}
\textsc{Painter, K.} and \textsc{Hillen, T.}:
\newblock {\em Volume-filling and quorum-sensing in models for chemosensitive
  movement}.
\newblock Can. Appl. Math. Q., 10(4):501--544, 2002.

\bibitem{SenbaSuzukiQuasilinearParabolicSystem2006}
\textsc{Senba, T.} and \textsc{Suzuki, T.}:
\newblock {\em A quasi-linear parabolic system of chemotaxis}.
\newblock Abstr. Appl. Anal.,
  \href{https://doi.org/10.1155/AAA/2006/23061}{2006:1--21}, 2006.
\newblock

\bibitem{souplet2019blow}
\textsc{Souplet, {\relax Ph}.} and \textsc{Winkler, M.}:
\newblock {\em Blow-up profiles for the parabolic-elliptic {{Keller--Segel}}
  system in dimensions $n \ge 3$}.
\newblock Commun. Math. Phys.,
  \href{https://doi.org/10.1007/s00220-018-3238-1}{367:665--681}, 2018.
\newblock

\bibitem{stinner2024critical}
\textsc{Stinner, {\relax Ch}.} and \textsc{Winkler, M.}:
\newblock {\em A critical exponent in a quasilinear {{Keller}}--{{Segel}}
  system with arbitrarily fast decaying diffusivities accounting for
  volume-filling effects}.
\newblock J. Evol. Equ.,
  \href{https://doi.org/10.1007/s00028-024-00954-x}{24(2):Paper No. 26}, 2024.
\newblock

\bibitem{TaoWinklerBoundednessQuasilinearParabolic2012}
\textsc{Tao, Y.} and \textsc{Winkler, M.}:
\newblock {\em Boundedness in a quasilinear parabolic{\textendash}parabolic
  {{Keller}}{\textendash}{{Segel}} system with subcritical sensitivity}.
\newblock J. Differ. Equ.,
  \href{https://doi.org/10.1016/j.jde.2011.08.019}{252(1):692--715}, 2012.
\newblock

\bibitem{WinklerDoesVolumefillingEffect2009}
\textsc{Winkler, M.}:
\newblock {\em Does a `volume-filling effect' always prevent chemotactic
  collapse?}
\newblock Math. Methods Appl. Sci.,
  \href{https://doi.org/10.1002/mma.1146}{33(1):12--24}, 2009.
\newblock

\bibitem{WinklerAggregationVsGlobal2010}
\textsc{Winkler, M.}:
\newblock {\em Aggregation vs. global diffusive behavior in the
  higher-dimensional {{Keller}}{\textendash}{{Segel}} model}.
\newblock J. Differ. Equ.,
  \href{https://doi.org/10.1016/j.jde.2010.02.008}{248(12):2889--2905}, 2010.
\newblock

\bibitem{WinklerFinitetimeBlowupHigherdimensional2013}
\textsc{Winkler, M.}:
\newblock {\em Finite-time blow-up in the higher-dimensional
  parabolic{\textendash}parabolic {{Keller}}{\textendash}{{Segel}} system}.
\newblock J. Math{\'e}matiques Pures Appliqu{\'e}es,
  \href{https://doi.org/10.1016/j.matpur.2013.01.020}{100(5):748--767}, 2013.
\newblock

\bibitem{WinklerCriticalBlowupExponent2018}
\textsc{Winkler, M.}:
\newblock {\em A critical blow-up exponent in a chemotaxis system with
  nonlinear signal production}.
\newblock Nonlinearity,
  \href{https://doi.org/10.1088/1361-6544/aaaa0e}{31(5):2031--2056}, 2018.
\newblock

\bibitem{WinklerFinitetimeBlowupLowdimensional2018}
\textsc{Winkler, M.}:
\newblock {\em Finite-time blow-up in low-dimensional {{Keller}}--{{Segel}}
  systems with logistic-type superlinear degradation}.
\newblock Z. F{\"u}r Angew. Math. Phys.,
  \href{https://doi.org/10.1007/s00033-018-0935-8}{69(2):Art.~40}, 2018.
\newblock

\bibitem{WinklerGlobalClassicalSolvability2019}
\textsc{Winkler, M.}:
\newblock {\em Global classical solvability and generic infinite-time blow-up
  in quasilinear {{Keller}}{\textendash}{{Segel}} systems with bounded
  sensitivities}.
\newblock J. Differ. Equ.,
  \href{https://doi.org/10.1016/j.jde.2018.12.019}{266(12):8034--8066}, 2019.
\newblock

\bibitem{WinklerBlowupProfilesLife2020}
\textsc{Winkler, M.}:
\newblock {\em Blow-up profiles and life beyond blow-up in the fully parabolic
  {{Keller-Segel}} system}.
\newblock J. Anal. Math{\'e}matique,
  \href{https://doi.org/10.1007/s11854-020-0109-4}{141(2):585--624}, 2020.
\newblock

\bibitem{WinklerFamilyMasscriticalKellerSegel2022}
\textsc{Winkler, M.}:
\newblock {\em A family of mass-critical {{Keller-Segel}} systems}.
\newblock Proc. Lond. Math. Soc. Third Ser.,
  \href{https://doi.org/10.1112/plms.12425}{124(2):133--181}, 2022.
\newblock

\bibitem{WinklerDjieBoundednessFinitetimeCollapse2010}
\textsc{Winkler, M.} and \textsc{Djie, K.~C.}:
\newblock {\em Boundedness and finite-time collapse in a chemotaxis system with
  volume-filling effect}.
\newblock Nonlinear Anal. Theory Methods Appl.,
  \href{https://doi.org/10.1016/j.na.2009.07.045}{72(2):1044--1064}, 2010.
\newblock

\end{thebibliography}
\end{document}